\documentclass{article}

\usepackage[english]{babel}
\usepackage{amsmath}
\usepackage{amssymb}
\usepackage{amsthm}
\usepackage{color}
\usepackage{mathtools}
\usepackage{verbatim}

\newtheorem{theorem}{Theorem}[section]
\newtheorem{cor}[theorem]{Corollary}
\newtheorem{defn}[theorem]{Definition}
\newtheorem{ex}[theorem]{Example}
\newtheorem{lem}[theorem]{Lemma}
\newtheorem{proposition}[theorem]{Proposition}
\newtheorem{rem}[theorem]{Remark}

\makeatletter
\newcommand{\leqnomode}{\tagsleft@true\let\veqno\@@leqno}
\newcommand{\reqnomode}{\tagsleft@false\let\veqno\@@eqno}
\makeatother

\title{Risk-Averse Models in Bilevel Stochastic Linear Programming}

\author{J. Burtscheidt \footnotemark[1]\
\and M. Claus \footnotemark[1]\
\and S. Dempe \footnotemark[2]}

\begin{document}

\maketitle

\renewcommand{\thefootnote}{\fnsymbol{footnote}}
\footnotetext[1]{Faculty of Mathematics, University of Duisburg-Essen, Campus Essen, Thea-Leymann-Str. 9,
D-45127 Essen, Germany, [johanna.burtscheidt][matthias.claus]@uni-due.de}
\footnotetext[2]{Faculty of Mathematics and Computer Science, TU Bergakademie Freiberg, Akademiestra\ss e 6, D-09599 Freiberg, Germany, dempe@tu-freiberg.de}
\renewcommand{\thefootnote}{\arabic{footnote}}

\begin{abstract}
We consider bilevel linear problems, where some parameters are stochastic, and the leader has to decide in a here-and-now fashion, while the follower has complete information. In this setting, the leader's outcome can be modeled by a random variable, which we evaluate based on some law-invariant convex risk measure. A qualitative stability result under perturbations of the underlying probability distribution is presented. Moreover, for the expectation, the expected excess, and the upper semideviation, we establish Lipschitz continuity as well as sufficient conditions for differentiability. Finally, for finite discrete distributions, we reformulate the bilevel stochastic problems as standard bilevel problems and propose a regularization scheme for bilevel linear problems.

\medskip


\medskip


\end{abstract}

\section{Introduction} \label{SecIntroduction}

Bilevel problems arise from the interplay of two decision makers at different levels of a hierarchy. The \emph{leader} decides first and passes the \emph{upper level decision} on to the \emph{follower}. Incorporating the leader's decision as a parameter, the follower then solves the \emph{lower level problem} reflecting his or her own goals and returns an optimal solution back to the leader. The leader's outcome depends on both his or her decision and the solution that is returned from the lower level. In bilevel optimization, it is assumed that the leader has full information about the influence of his or her decision on the lower level problem. As the latter may have more than one solution, models typically consider the case where the follower returns either the best (\emph{optimistic model}) or the worst (\emph{pessimistic model}) solution with respect to the leader's objective. The bilevel optimization problem is to find an optimal upper level decision which, even in a linear setting, results in a nonconvex, nondifferentiable and NP-hard problem (cf. \cite[Chapter 3]{Dempe2002}).

\medskip

The present work is on bilevel stochastic linear problems, where the realization of some random vector whose distribution does not depend on the upper level decision enters the lower level problem as an additional parameter. It is assumed that the leader has to make his or her decision without knowing the realization of the randomness, while the follower decides under full information. This setting encapsulates two-stage stochastic programming with linear recourse as the special case, where the upper and lower level objective functions coincide.

\medskip

In classical two-stage stochastic programming, the upper level objective function gives rise to a family of random variables defined by the optimal value function of the recourse problem. In contrast, the arising random variables in optimistic bilevel stochastic programming models depend on the optimal value of a problem where only optimal solutions of the lower level problem are feasible and the decision is made by a different actor. This is a crucial difference that entails a loss of convexity  and poses additional challenges.

\medskip

Nevertheless, bilevel stochastic problems are of great relevance for practical applications and have been discussed in the context of pricing of electricity swing options (\cite{KovacevicPflug2013}), economics (\cite{CarrionArroyoConejo2009}), supply chain planning (\cite{RoghanianSadjadiAryanezhad2007}), telecommunications (\cite{Werner2005}) and general agency problems (\cite{GaivoronskiWerner2012}). Other works focus on solution methods (\cite{BirbilGuerkanListes2004}), bilevel stochastic problems with Knapsack constraints (\cite{KosuchLeBodicLeungLisser2012}) and SMPECs (\cite{LinFukushima2010}).

\medskip

In \cite{Ivanov2014}, Ivanov examines bilevel stochastic linear problems with uncertainty in the right-hand side of the lower level problem and utilizes the Value-at-Risk to rank the arising random variables. The results include continuity of the objective function, the existence of a solution, and equivalence to a mixed-integer linear program, if the underlying distribution is finite discrete. The latter result has been extended to the fully random case in \cite{DempeIvanovNaumov2017}.

\medskip

In the present work, we rank the random variables arising from right-hand side uncertainty in the lower level by law-invariant risk measures. In particular, we consider the expectation, the expected excess over a fixed target level, the mean upper semideviation and the Conditional Value-at-Risk and establish Lipschitz continuity of the resulting objective function.

\medskip

It is well known that stochastic programming models may be smoother than their underlying deterministic counterparts. For instance, for a class of stochastic Stackelberg games employing the expectation, differentiability has been derived in \cite{DeMiguelXu2009}. Overcoming additional challenges arising from nondifferentiable integrands, we establish continuous differentiability for bilevel stochastic linear problems using the expectation, the expected excess or the mean upper semideviation.

\medskip

Incomplete information or the need for computational efficiency may lead to optimization models where an approximation of the true underlying distribution is employed. This motivates the analysis of the behavior of optimal values and (local) optimal solution sets under perturbations of the underlying distribution (see e.g. \cite{LiuXuLin2011}, \cite{Patriksson2008} and \cite{PatrikssonWynter1999} for stability analysis of related models). For bilevel stochastic linear problems, we establish a qualitative stability result that holds for all law-invariant convex risk measures.

\medskip

All our results regarding finiteness, (Lipschitz) continuity, differentiability and stability cover both the optimistic and the pessimistic approach of bilevel stochastic linear programming.

\medskip

For finite discrete distributions and optimistic models, we show that the risk-averse bilevel stochastic linear problems using the expectation, the expected excess or the mean upper semideviation are equivalent to standard bilevel linear problems. The resulting problems for the expectation and expected excess have at most one coupling constraint involving variables from different scenarios, which paves the way for decomposition approaches.

\medskip

Finally, we show that a simplified version of the regularization scheme in \cite{Scholtes2001} can be used to solve bilevel linear problems.

\section{Model} \label{SecModel}

Using the optimistic model, we shall consider parametric bilevel linear problems of the form
$$
\min_x \left\{c^\top x + \min_y \{q^\top y \; | \; y \in \Psi(x,z)\} \; | \; x \in X \right\},
$$
where $X \subseteq \mathbb{R}^n$ is nonempty, $c \in \mathbb{R}^n$ and $q \in \mathbb{R}^m$ are vectors, and $\Psi: \mathbb{R}^n \times \mathbb{R}^s \rightrightarrows \mathbb{R}^m$ is the lower level optimal solution set mapping defined by
\begin{equation*}
	\Psi(x,z) := \underset{y}{\mathrm{Argmin}} \; \{d^\top y \; | \; Ay \leq Tx + z\}
\end{equation*}
with matrices $A \in \mathbb{R}^{s \times m}$, $T \in \mathbb{R}^{s \times n}$ and a vector $d \in \mathbb{R}^m$. A bilevel stochastic program arises if we assume that the parameter $z = Z(\omega)$ is the realization of a known random vector $Z$ defined on some probability space $(\Omega, \mathcal{F}, \mathbb{P})$. We impose an additional non-anticipativity constraint that creates the following pattern of decision and observation:
\medskip
\begin{center}
	Leader decides $x$ \hspace{15pt} $\rightarrow$ \hspace{15pt} $z = Z(\omega)$ is revealed \hspace{15pt} $\rightarrow$ \hspace{15pt} Follower decides $y$.
\end{center}
\medskip
Throughout the analysis, we assume the stochasticity to be purely exogenous, i.e. the distribution of $Z$ to be independent of $x$. In this setting, the leader's decision $x$ gives rise to the random variable
\begin{equation*}
	f(x,Z(\cdot)) := c^\top x + \min_y \{q^\top y \; | \; y \in \Psi(x,Z(\cdot))\}
\end{equation*}
and the problem can be understood as picking an optimal random variable from the family $f(X,Z) := \{f(x, Z(\cdot)) \; | \; x \in X\} \subseteq L^0(\Omega,\mathcal{F},\mathbb{P})$. We shall rank these random variables according to some mapping $\mathcal{R}: L^0(\Omega,\mathcal{F},\mathbb{P}) \to \mathbb{R} \cup \{\pm\infty\} =: \overline{\mathbb{R}}$, i.e. consider the bilevel stochastic problem
\begin{equation}
\label{ELB}
	\min_x \left\{\mathcal{R}[f(x,Z(\cdot))] \; | \; x \in X \right\}.
\end{equation}
We shall assume that there is some $p \in [1,\infty)$ such that the restriction $\mathcal{R}|_{L^p(\Omega, \mathcal{F},\mathbb{P})}$ is real-valued, convex, nondecreasing w.r.t. the $\mathbb{P}$-almost sure partial order. Furthermore, let $\mathcal{R}$ be law-invariant, i.e. $\mathcal{R}[Y] = \mathcal{R}[Y']$ whenever the induced Borel measures $\mathbb{P} \circ Y$ and $\mathbb{P} \circ Y'$ coincide.

\medskip

\begin{rem}
The above assumptions are fulfilled for any law-invariant convex risk measure in the sense of \cite{FoellmerSchied2002, FrittelliGianin2002} (see also \cite{FoellmerSchied2011, FrittelliGianin2005}). However, we do not assume translation equivariance for the present analysis.
\end{rem}

\medskip

\begin{ex} \label{ExConvexRM}

\medskip

\begin{enumerate}
	\item The expectation $\mathbb{E}[\cdot]$,
	
	\medskip
	
	\item the expected excess $\mathrm{EE}_\eta[\cdot] =  \mathbb{E}[\max \lbrace \cdot- \eta, 0\}]$ over a fixed target level $\eta \in \mathbb{R}$,
	
	\medskip

	\item any weighted sum $\mathrm{SD}_\rho[\cdot] = \mathbb{E}[\cdot] + \rho \mathrm{EE}_{E[\cdot]}[\cdot]$ of the expectation and the upper semideviation with $\rho \in [0,1)$ and

	\medskip	
	
	\item the Conditional Value at Risk
	$$
	\mathrm{CVaR}_{\alpha}[\cdot] = \min_{\eta \in \mathbb{R}} \left\{ \eta \; + \; \frac{1}{1 - \alpha} \mathrm{EE}_\eta[\cdot] \right\}
	$$
	for a fixed level $\alpha \in (0,1)$ (cf. \cite[Theorem 10]{RockafellarUryasev2002})
\end{enumerate}

\medskip

are law-invariant and fulfill the above assumptions (see e.g. \cite{DentchevaRuszczynskiShapiro2014}, \cite{Pflug2000}). In all of the above situations $p$ can be chosen as $1$.
\end{ex}

\section{Structural properties} \label{SecStructure}

In this section, we shall consider the case where $\mathcal{R}$ is given by the $\mathbb{E}$, $\mathrm{EE}_\eta$ or $\mathrm{SD}_\rho$ and examine properties of the mapping $Q_\mathcal{R}: \mathbb{R}^{n} \to \overline{\mathbb{R}}$ given by
$$
Q_\mathcal{R}(x) := \mathcal{R}[f(x,Z(\cdot)].
$$
First, we shall prove that the function $f$ defined above is Lipschitz continuous and hence Borel measurable.

\medskip

\begin{lem} \label{LemmaF}
Assume that $\mathrm{dom} \; f \neq \emptyset$, then $f$ is real-valued and Lipschitz continuous on the polyhedron
$F = \{(x,z) \in \mathbb{R}^n \times \mathbb{R}^s \; | \; \exists y \in \mathbb{R}^m: Ay \leq Tx + z\}$.
\end{lem}

\medskip

\begin{proof}
By \cite{Eaves1971}, $\emptyset \neq \mathrm{dom} \; f \subseteq \mathrm{dom} \; \Psi$ implies $\mathrm{dom} \; \Psi = F$. Consequently, the linear program in the definition of $f(x,z)$ is solvable for any $(x,z) \in F$ by parametric linear programming theory (see \cite{Beer1977}). Consider any $(x,z), (x',z') \in F$. Without loss of generality, assume that $f(x,z) \geq f(x',z')$ and let $y' \in \Psi(x',z')$ be such that $f(x',z') = c^\top x' + q^\top y'$. Following \cite{KlatteKummer1984} we obtain
\begin{align*}
|f(x,z)-f(x',z')| \; &= \; f(x,z)- c^\top x' - q^\top y' \; \leq \; c^\top x + q^\top y - c^\top x' - q^\top y' \\
&\leq \; \|c\| \|x-x'\| + \|q\| \|y-y'\|
\end{align*}
for any $y \in \Psi(x,z)$. Let $\mathbb{B}$ denote the Euclidean unit ball, then Theorem \ref{TheoremKlatteRightHandSide} in the Appendix yields
\begin{equation*}
\Psi(x',z')\subseteq \Psi(x,z)+ \Lambda \|(x,z)-(x',z')\|\mathbb{B}
\end{equation*}
and hence $|f(x,z)-f(x',z')| \; \leq \; (\|c\| + \Lambda \|q\|) \|(x,z)-(x',z')\|$.
\end{proof}

\medskip

\begin{rem}
In view of Theorem \ref{TheoremKlatteRightHandSide} in the Appendix, the above result can be easily extended to the case of a convex quadratic lower level problem.
\end{rem}

\medskip

The next result follows directly from linear programming theory and provides verifiable conditions for $\mathrm{dom} \; f \neq \emptyset$:

\medskip

\begin{lem}
$\mathrm{dom} \; f \neq \emptyset$ holds if and only if there exists $(x,z) \in \mathbb{R}^n \times \mathbb{R}^s$ such that

\medskip

\begin{enumerate}
	\item $\{y \; | \;  Ay \leq Tx + z\}$ is nonempty,
	
	\smallskip
	
	\item there is some $u \in \mathbb{R}^s$ satisfying $A^\top u = d$ and $u \leq 0$, and
	
	\smallskip
	
	\item the function $y \mapsto q^\top y$ is bounded from below on $\Psi(x,z)$.
\end{enumerate}

\smallskip

Under these conditions,
$$
\min_{y} \lbrace q^\top y \; | \; y \in \Psi(x',z') \rbrace
$$
is attained for any $(x',z') \in F$.
\end{lem}

\medskip

Under an appropriate moment condition, Lemma \ref{LemmaF} implies finiteness and Lipschitz continuity of $Q_{\mathbb{E}}$, $Q_{\mathrm{EE}_\eta}$ and $Q_{\mathrm{SD}_\rho}$. Let
$$
\mathcal{M}^p_s := \left\{ \mu \in \mathcal{P}(\mathbb{R}^s) \; | \; \int_{\mathbb{R}^s} \|z\|^p~\mu(dz) < \infty \right\}
$$
denote the set of Borel probability measures on $\mathbb{R}^s$ with finite moments of order $p \in [0,\infty)$.

\medskip

\begin{theorem} \label{ThFiniteLipschitz}
Assume $\mathrm{dom} \; f \neq \emptyset$ and $\mu_Z := \mathbb{P} \circ Z^{-1} \in \mathcal{M}^1_s$. Then the mappings $Q_{\mathbb{E}}$, $Q_{\mathrm{EE}_\eta}$, $Q_{\mathrm{SD}_\rho}$ and $Q_{\mathrm{CVaR}_\alpha}$ are real-valued and Lipschitz continuous on
$$
F_Z = \lbrace x \in \mathbb{R}^n \; | \; (x,z) \in F \; \forall z \in \mathrm{supp} \; \mu_Z \rbrace
$$
for any $\eta \in \mathbb{R}$, $\rho \in [0,1)$ and $\alpha \in (0,1)$.
\end{theorem}

\medskip

\begin{proof}
$Q_\mathbb{E}$: Let $L$ be the Lipschitz constant from Lemma \ref{LemmaF}. For any  $z_0 \in \mathrm{supp} \; \mu_Z$ and $x \in F_Z$ we have
\begin{align*}
|Q_\mathbb{E}(x)| \; &= \; \left| \int_{\mathrm{supp} \; \mu_Z} f(x,z)~\mu_Z(dz) \right| \\
&\leq \; |f(x,z_0)| + \int_{\mathrm{supp} \; \mu_Z} |f(x,z)-f(x,z_0)|~\mu_Z(dz) \\
&\leq \; |f(x,z_0)| + L\|z_0\| + \int_{\mathrm{supp} \; \mu_Z} L\|z\|~\mu_Z(dz) < \infty.
\end{align*}
Furthermore,
\begin{align*}
|Q_\mathbb{E}(x) -Q_\mathbb{E}(x')| \; \leq \; \int_{\mathrm{supp} \; \mu} |f(x,z)-f(x',z)|~\mu(dz) \; \leq \; L \|x-x'\|
\end{align*}
holds for any $x,x' \in F_Z$.

\medskip

$Q_{\mathrm{EE}_\eta}$: Invoking $\max \lbrace f(x,z) - \eta, 0 \rbrace \leq |f(x,z)| + |\eta|$ and the Lipschitz continuity of $x \mapsto \max \lbrace f(x,z) - \eta, 0 \rbrace$ on $F_Z$, finiteness and Lipschitz continuity of $Q_{\mathrm{EE}_\eta}$ can be shown by the same arguments as for $Q_\mathbb{E}$.

\medskip

$Q_{\mathrm{SD}_\rho}$: Finiteness and Lipschitz continuity follow from the corresponding results for $Q_\mathbb{E}$ and $Q_{\mathrm{EE}_\eta}$.

\medskip

$Q_{\mathrm{CVaR}_\alpha}$: Consider the mapping $g: \mathbb{R}^n \to L^0(\Omega, \mathcal{F}, \mathbb{P})$, $g(x) := f(x,Z(\cdot))$. By the results for $Q_\mathbb{E}$, we have $g(F_Z) \subseteq L^1(\Omega, \mathcal{F}, \mathbb{P})$ and the restriction $g|_{F_Z}$ is Lipschitz continuous w.r.t. the $L^1$-norm. Consequently, the composition $Q_{\mathrm{CVaR}_\alpha} = \mathrm{CVaR}_\alpha \circ g$ is finite an Lipschitz continuous on $F_Z$ by \cite[Corollary 3.7 and the subsequent remark]{Pichler2017}.

\end{proof}

\medskip

Under the assumptions of Theorem \ref{ThFiniteLipschitz}, the bilevel stochastic linear problem is solvable whenever $X$ is a nonempty compact subset of $F_Z$. A similar result holds for a comprehensive class of risk measure and shall be discussed in Section \ref{SecStability} (cf. Corollary \ref{CorCont}).

\medskip

We shall now focus on differentiability of $Q_\mathcal{R}$. It will be convenient to reformulate $f$ as
$$
f(x,z) = c^\top x + \min_{y_+,y_-,t} \big\{ q^\top(y_{+} - y_{-}) \; | \; (y_+,y_-,t) \in \Psi_{=}(x,z) \big\},
$$
where
$$
\Psi_{=}(x,z) = \underset{y_+,y_-,t}{\mathrm{Argmin}} \lbrace d^\top(y_{+} - y_{-})\; | \; A(y_{+} - y_{-}) + t = Tx + z, \; y_+,y_-,t \geq 0 \rbrace.
$$
Setting
$$
\hat{q} := \begin{pmatrix} q \\ -q \\ 0_s \end{pmatrix}, \; \hat{y} := \begin{pmatrix} y_+ \\ y_- \\ t \end{pmatrix}, \; \hat{d} := \begin{pmatrix} d \\ -d \\ 0_s \end{pmatrix}, \; \text{and} \; \hat{A} := (A,-A,I_s)
$$
we obtain
$$
f(x,z) = c^\top x + \min_{\hat{y}} \big\{ \hat{q}^\top \hat{y} \; | \; \hat{y} \in \underset{y'}{\mathrm{Argmin}} \lbrace \hat{d}^\top y' \; | \; \hat{A}y' = Tx+z, \; y' \geq 0 \rbrace \big\}.
$$
As the rows of $\hat{A}$ are linearly independent, we may consider the nonempty set
$$
\mathcal{A} := \lbrace \hat{A}_B \in \mathbb{R}^{s \times s} \; | \; \hat{A}_B \; \text{is a regular submatrix of } \hat{A} \rbrace
$$
of lower level base matrices. A base matrix $\hat{A}_B \in \mathcal{A}$ is optimal for the lower level problem for a given $(x,z)$ if it is feasible, i.e. $\hat{A}_B^{-1}(Tx+z) \geq 0$, and the associated reduced cost vector $\hat{d}_N^\top - \hat{d}_B^\top  \hat{A}_B^{-1} \hat{A}_N$ is nonnegative. Furthermore, for any optimal base matrix $\hat{A}_{B'} \in \mathcal{A}$, there exists a feasible base matrix $\hat{A}_B \in \mathcal{A}$ satisfying
$$
\hat{A}_{B'}^{-1}(Tx+z) = \hat{A}_{B}^{-1}(Tx+z) \; \; \text{and} \; \; \hat{d}_N^\top - \hat{d}_B^\top  \hat{A}_B^{-1} \hat{A}_N \geq 0.
$$
Set
$$
\mathcal{A}^\ast := \lbrace \hat{A}_B \in \mathcal{A} \; | \; \hat{d}_N^\top - \hat{d}_B^\top  \hat{A}_B^{-1} \hat{A}_N \geq 0 \rbrace
$$
and assume $\mathrm{dom} \; f \neq \emptyset$, then
\begin{equation}
\label{fBasisRef}
f(x,z) = c^\top x + \min_{\hat{A}_B} \big\{ \hat{q}^\top_B \hat{A}_B^{-1}(Tx+z) \; | \; \hat{A}_B^{-1}(Tx+z) \geq 0, \; \hat{A}_B \in \mathcal{A}^\ast \big\}
\end{equation}
holds for any $(x,z) \in F$.

\medskip

\begin{defn}
The \emph{region of stability} associated with a base matrix $\hat{A}_{B} \in \mathcal{A}^\ast$ is the set
$$
\mathcal{R}(\hat{A}_B) := \lbrace (x,z) \in F \; | \; \hat{A}_B^{-1}(Tx + z) \geq 0, \; c^\top x + \hat{q}^\top_B \hat{A}_B^{-1}(Tx+z) = f(x,z) \rbrace.
$$
\end{defn}

\begin{lem} \label{LemmafDiff}
Assume $\mathrm{dom} \; f \neq \emptyset$ and let $x_0$ be an inner point of $F_Z$. Then $f(\cdot,z_0)$ is continuously differentiable at $x_0$ for any $z_0 \in \mathrm{supp} \; \mu_Z \setminus \mathcal{N}_{x_0}$, where
$$
\mathcal{N}_{x_0} = \mathcal{F}_{x_0} \; \cup \;  \bigcup_{\hat{A}_B \in \mathcal{A}^\ast} \mathcal{Z}_{x_0}(\hat{A}_B) \setminus \mathrm{int} \; \mathcal{Z}_{x_0}(\hat{A}_B) \; \cup \bigcup_{\substack{\hat{A}_B, \hat{A}_{B'} \in \mathcal{A}^\ast: \\ \hat{q}^\top_B \hat{A}_B^{-1} \; \neq \; \hat{q}^\top_{B'} \hat{A}_{B'}^{-1}}}  \mathcal{V}_{x_0}(\hat{A}_B, \hat{A}_{B'})
$$
with
\begin{align*}
\mathcal{F}_{x_0} &:= \lbrace z \in \mathbb{R}^s \; | \; (x_0,z) \in F \setminus \mathrm{int} \; F \rbrace, \\
\mathcal{Z}_{x_0}(\hat{A}_B) &:= \lbrace z \in \mathbb{R}^s \; | \; \hat{A}_B^{-1}(Tx_0 + z) \geq 0 \rbrace, \; \text{and} \\
\mathcal{V}_{x_0}(\hat{A}_B, \hat{A}_{B'}) &:=  \lbrace z \in \mathbb{R}^s \; | \; \hat{q}^\top_B (\hat{A}_B^{-1} - \hat{A}_{B'}^{-1})(z + Tx_0) = 0 \rbrace.
\end{align*}
Furthermore, $\mathcal{N}_{x_0}$ is contained in a finite union of affine hyperplanes in $\mathbb{R}^s$ and we have
$$
\nabla_x f(x_0,z_0) \in \lbrace c^\top + \hat{q}^\top_B \hat{A}_B^{-1} T \; | \; \hat{A}_B \in \mathcal{A}^\ast \rbrace.
$$
\end{lem}

\begin{proof}
$x_0 \in \mathrm{int} \; F_Z$ and $z_0 \in \mathrm{supp} \; \mu \setminus \mathcal{N}_{x_0} \subseteq \mathrm{supp} \; \mu \setminus \mathcal{F}_{x_0}$ imply $(x_0,z_0) \in \mathrm{int} \; F$ by definition.
In view of \eqref{fBasisRef}, we have
$$
F \subseteq \bigcup_{\hat{A}_B \in \mathcal{A}^\ast} \mathcal{R}(\hat{A}_B).
$$
If $(x_0,z_0) \in \mathrm{int} \; \mathcal{R}(\hat{A}_B)$ holds for some $\hat{A}_B \in \mathcal{A}^\ast$, there is a neighborhood $U$ of $x_0$ such that $f(x,z_0) = c^\top x + \hat{q}^\top_B \hat{A}_B^{-1}(Tx+z_0)$ holds for all $x \in U$. In particular, $f(\cdot,z_0)$ is continuously differentiable at $x_0$ and $\nabla_x f(x_0,z_0) = c^\top + \hat{q}^\top_B \hat{A}_B^{-1} T$.

\medskip

Suppose that $(x_0,z_0) \notin \mathrm{int} \; \mathcal{R}(\hat{A}_B)$ for all $\hat{A}_B \in \mathcal{A}^\ast$. The continuity of $f$ implies that there are $k \geq 2$ pairwise different base matrices $\hat{A}_{B^1}, \ldots, \hat{A}_{B^k} \in \mathcal{A}^\ast$ such that
$$
(x_0,z_0) \in \bigcap_{i=1,\ldots,k} \mathcal{R}(\hat{A}_{B^i}) \; \cap \; \mathrm{int} \left( \bigcup_{i=1,\ldots,k} \mathcal{R}(\hat{A}_{B^i}) \right).
$$
In particular, we have
$$
\hat{q}^\top_{B^1} \hat{A}_{B^1}^{-1}(Tx_0+z_0) = \ldots = \hat{q}^\top_{B^k} \hat{A}_{B^k}^{-1}(Tx_0+z_0),
$$
i.e. $z_0 \in \mathcal{V}_{x_0}(\hat{A}_{B^i}, \hat{A}_{B^j})$ for all $i,j \in \lbrace 1, \ldots, k \rbrace$. Thus, $z_0 \in \mathrm{supp} \; \mu \setminus \mathcal{N}_{x_0}$ implies
\begin{equation}
\label{ProofLemfDiff}
\hat{q}^\top_{B^1} \hat{A}_{B^1}^{-1} = \ldots = \hat{q}^\top_{B^k} \hat{A}_{B^k}^{-1}.
\end{equation}
For any $i \in \lbrace 1, \ldots, k \rbrace$ we shall consider the sets
\begin{align*}
\mathcal{Z}(\hat{A}_{B^i}) &:= \lbrace (x,z) \in \mathbb{R}^n \times \mathbb{R}^s \; | \; \hat{A}_{B^i}^{-1}(Tx + z) \geq 0\rbrace \; \; \text{and} \\
\mathcal{O}(\hat{A}_{B^i}) &:= \lbrace (x,z) \in \mathbb{R}^n \times \mathbb{R}^s \; | \; c^\top x + \hat{q}^\top_{B^i} \hat{A}_{B^i}^{-1}(Tx+z) = f(x,z) \rbrace.
\end{align*}
By \eqref{ProofLemfDiff} we have
$$
\mathcal{R}(\hat{A}_{B^i}) = \mathcal{O}(\hat{A}_{B^i}) \cap \mathcal{Z}(\hat{A}_{B^i}) = \mathcal{O}(\hat{A}_{B^1}) \cap \mathcal{Z}(\hat{A}_{B^i})
$$
for all $i \in \lbrace 1, \ldots, k \rbrace$. Thus,
\begin{align*}
(x_0,z_0) &\in \mathrm{int} \bigcup_{i=1,\ldots,k} \mathcal{R}(\hat{A}_{B^i}) \; = \; \mathrm{int} \left( \mathcal{O}(\hat{A}_{B^1}) \; \cap \bigcup_{i=1,\ldots,k} \mathcal{Z}(\hat{A}_{B^i}) \right) \\
&= \mathrm{int} \; \mathcal{O}(\hat{A}_{B^1}) \; \cap \; \mathrm{int} \bigcup_{i=1,\ldots,k} \mathcal{Z}(\hat{A}_{B^i}).
\end{align*}
We have $(x_0,z_0) \in \mathcal{Z}(\hat{A}_{B^1})$, i.e. $z_0 \in \mathcal{Z}_{x_0}(\hat{A}_{B^1})$. Thus, $z_0 \in \mathrm{supp} \; \mu \setminus \mathcal{N}_{x_0}$ implies \\ $z_0 \in \mathrm{int} \; \mathcal{Z}_{x_0}(\hat{A}_{B^1})$. Consequently, there is a neighborhood $W$ of $z_0$ such that $\hat{A}_{B^1}^{-1}(Tx_0 + z) \geq 0$ for all $z \in W$. This implies $(x_0,z_0) \in \mathrm{int} \; \mathcal{Z}(\hat{A}_{B^1})$ for continuity reasons. Hence, $(x_0,z_0) \in \mathrm{int} \; \mathcal{O}(\hat{A}_{B^1}) \; \cap \; \mathrm{int} \; \mathcal{Z}(\hat{A}_{B^1}) = \mathrm{int} \; \mathcal{R}(\hat{A}_{B^1})$, which contradicts $(x_0,z_0) \notin \mathrm{int} \; \mathcal{R}(\hat{A}_B)$ for all $\hat{A}_B \in \mathcal{A}^\ast$.

\medskip

It remains to show that $\mathcal{N}_{x_0}$ is contained in a finite union of affine hyperplanes. Suppose that $z$ is such that $Ay < Tx_0 + z$ holds for some $y \in \mathbb{R}^m$, then $(x_0,z) \in \mathrm{int} \; F$. Consequently,
$$
\mathcal{F}_{x_0} \subseteq \bigcup_{i=1,\ldots,s} \lbrace z \in \mathbb{R}^s \; | \; (x,z_0) \in F, \; e_i^\top Ay = e_i^\top(Tx_0 + z) \; \forall y: \; Ay \leq Tx_0 + z \rbrace
$$
is contained in a finite union of affine hyperplanes. Similarly we have
$$
\mathcal{Z}_{x_0}(\hat{A}_B) \setminus \mathrm{int} \; \mathcal{Z}_{x_0}(\hat{A}_B) \subseteq \bigcup_{i=1,\ldots,s} \lbrace z \in \mathbb{R}^s \; | \; \hat{A}_B^{-1}(Tx_0 + z) \leq 0, \; e_i^\top(\hat{A}_B^{-1}(Tx_0 + z)) = 0 \rbrace
$$
$\hat{A}_B \in \mathcal{A}^\ast$, where $e_i^\top \hat{A}_B^{-1} \neq 0$ due to the regularity of $\hat{A}_B$. Finally, $\mathcal{V}_{x_0}(\hat{A}_B, \hat{A}_{B'})$ is an affine hyperplane for any $\hat{A}_B, \hat{A}_{B'} \in \mathcal{A}^\ast$ satisfying $\hat{q}^\top_B \hat{A}_B^{-1} \neq \hat{q}^\top_{B'} \hat{A}_{B'}^{-1}$.
\end{proof}

\medskip

\begin{theorem} \label{PropExpDiff}
Assume $\mathrm{dom} \; f \neq \emptyset$, $\mu_Z \in \mathcal{M}^1_s$, and let $x_0 \in \mathrm{int} \; F_Z$ be such that $\mu_Z[\mathcal{N}_{x_0}] = 0$. Then $Q_{\mathbb{E}}$ is continuously differentiable at $x_0$ and
$$
Q'_{\mathbb{E}}(x_0) = \int_{\mathrm{supp} \; \mu_Z \setminus \mathcal{N}_{x_0}} \nabla_x f(x_0,z)~\mu_Z(dz).
$$
\end{theorem}

\medskip

\begin{proof}
We shall prove that Lemma \ref{LemmaDiffInt} in the Appendix is applicable. First, note that condition (a) is satisfied by $\mu_Z[\mathcal{N}_{x_0}] = 0$ and Lemma \ref{LemmafDiff}. Furthermore, by $x_0 \in \mathrm{int} \; F_Z$ there is neighborhood $U$ of $x_0$ that is contained in $F_Z$. In particular, $Q_\mathbb{E}$ is well-defined and finite by Proposition \ref{ThFiniteLipschitz}, i.e. the first part of condition (b) of Lemma \ref{LemmaDiffInt} is satisfied. To see that the second part holds as well, let $L$ denote the Lipschitz constant from Lemma \ref{LemmaF}. Fix any $x \in U \setminus \lbrace x_0 \rbrace$ and $z_0 \in \mathrm{supp} \; \mu_Z \setminus \mathcal{N}_{x_0}$, then
$$
\|x-x_0\|^{-1} \big|f(x,z_0) - f(x_0,z_0) - \nabla_x f(x_0,z_0)(x-x_0)\big|  \;
\leq \; L + \max_{\hat{A}_{B} \in \mathcal{A}^\ast} \|\hat{q}_{B}^\top \hat{A}_B^{-1} B\|
$$
follows immediately from the characterization of the derivative in Lemma \ref{LemmafDiff}. Thus, Lemma \ref{LemmaDiffInt} yields the differentiability of $Q_\mathbb{E}$.

\medskip

We shall now prove that the derivative is indeed continuous. By construction, there exists a neighborhood $U \subseteq \mathrm{int} \; F_Z$ of $x_0$ such that $\mathcal{N}_x \subseteq \mathcal{N}_{x_0}$ holds for any $x \in U$. Consequently, by $\mu_{Z}[\mathcal{N}_{x}] = 0$ and the previous arguments, $Q_\mathbb{E}$ is differentiable at any $x \in U$ and we have
\begin{equation*}
Q'_{\mathbb{E}}(x) = \int_{\mathrm{supp} \; \mu_Z \setminus \mathcal{N}_{x}} \nabla_x f(x,z)~\mu_Z(dz) = c^\top + \sum_{\Delta \in D} \mu_Z[\mathcal{W}(x,\Delta)] \Delta,
\end{equation*}
where $D := \lbrace \hat{q}_B^\top \hat{A}_B^{-1} T \; | \; \hat{A}_B \in \mathcal{A}^\ast \rbrace$ and
$$
\mathcal{W}(x,\Delta) := \lbrace z \in \mathrm{supp} \; \mu_Z \setminus \mathcal{N}_{x} \; | \; \exists \hat{A}_B \in \mathcal{A}^\ast: (x,z) \in \mathrm{int} \;\mathcal{R}(\hat{A}_B), \; \Delta =  \hat{q}_B^\top \hat{A}_B^{-1} T \rbrace.
$$
By Lemma \ref{LemmaOuterSemicont} in the Appendix, the set-valued mapping $\overline{\mathcal{W}}: \mathbb{R}^n \times D \rightrightarrows \mathbb{R}^s$,
\begin{align*}
\overline{\mathcal{W}}(x,\Delta) &= \lbrace z \in \mathrm{supp} \; \mu_Z \; | \; \exists \hat{A}_B \in \mathcal{A}^\ast: (x,z) \in \mathrm{cl} \; \mathrm{int} \;\mathcal{R}(A_B), \; \Delta =  \hat{q}_B^\top \hat{A}_B^{-1} T \rbrace \\
&= \lbrace z \in \mathrm{supp} \; \mu_Z \; | \; (x,z) \in \bigcup_{\hat{A}_B \in \mathcal{A}^\ast: \; \hat{q}_B^\top \hat{A}_B^{-1} T = \Delta}  \mathrm{cl} \; \mathrm{int} \; \mathcal{R}(A_B) \rbrace
\end{align*}
is outer semicontinuous. Furthermore, by the arguments used in the proof of Lemma \ref{LemmafDiff} we obtain
$$
\overline{\mathcal{W}}(x,\Delta) \setminus \mathcal{W}(x,\Delta) \subseteq \mathcal{N}_x
$$
and thus $\mu_Z[\mathcal{N}_{x}] = 0$ implies $\mu_Z[\mathcal{W}(x,\Delta)] = \mu_Z[\overline{\mathcal{W}}(x,\Delta)]$.

\medskip

We shall use the above representation to prove that for any $\Delta \in D$, the mapping $M_\Delta: \mathbb{R}^n \to \mathbb{R}$, $M_\Delta(x) := \mu_Z[\mathcal{W}(x,\Delta)]$ is continuous at $x_0$. Consider any sequence $\lbrace x_l \rbrace_{l \in \mathbb{N}} \subset \mathbb{R}^n$ that converges to $x_0$. Without loss of generality we may assume that $x_l \in U$ holds for all $l \in \mathbb{N}$. We have
\begin{align*}
\limsup_{l \to \infty} M_{\Delta}(x_l) &= \limsup_{l \to \infty} \mu_Z[\overline{\mathcal{W}}(x_l,\Delta)] = \limsup_{l \to \infty} \int_{\overline{\mathcal{W}}(x_l,\Delta)} 1~\mu_Z(dz) \\
&= \limsup_{l \to \infty} \int_{\mathrm{supp} \; \mu_Z} \mathrm{1}_{\overline{\mathcal{W}}(x_l,\Delta)}(z)~\mu_Z(dz) \\
&\leq \int_{\mathrm{supp} \; \mu_Z} \limsup_{l \to \infty} \; \mathrm{1}_{\overline{\mathcal{W}}(x_l,\Delta)}(z)~\mu_Z(dz),
\end{align*}
where
$$
\mathrm{1}_{\overline{\mathcal{W}}(x_l,\Delta)} := \Big\{ \begin{matrix} 1, &\text{if $z \in \overline{\mathcal{W}}(x_l,\Delta)$} \\ 0, &\text{else}\end{matrix}
$$
denotes the indicator function associated with the set $\overline{\mathcal{W}}(x_l,\Delta)$ and the final inequality is obtained by using Fatou's Lemma. We shall show that
\begin{equation}
\label{LimsupIndicator}
\limsup_{l \to \infty} \mathrm{1}_{\overline{\mathcal{W}}(x_l,\Delta)}(z) \leq \mathrm{1}_{\limsup_{l \to \infty} \overline{\mathcal{W}}(x_l,\Delta)}(z)
\end{equation}
holds for any $z \in \mathrm{supp} \; \mu_Z$. If the left-hand side in \eqref{LimsupIndicator} equals zero, the above inequality holds because the right-hand side is nonnegative. On the other hand, $\limsup_{l \to \infty} \mathrm{1}_{\overline{\mathcal{W}}(x_l,\Delta)}(z) = 1$ implies that there is a subsequence $\lbrace x'_l \rbrace_{l \in \mathbb{N}}$ of $\lbrace x_l \rbrace_{l \in \mathbb{N}}$ such that $z \in \overline{\mathcal{W}}(x'_l,\Delta)$ holds for all $l \in \mathbb{N}$. Thus, $z \in \limsup_{l \to \infty} \overline{\mathcal{W}}(x_l,\Delta)$ by definition and \eqref{LimsupIndicator} is satisfied.

\medskip

Invoking \eqref{LimsupIndicator} and the previous estimates we obtain
\begin{align*}
\limsup_{l \to \infty} M_{\Delta}(x_l) &\leq \int_{\mathrm{supp} \; \mu_Z} \mathrm{1}_{\limsup_{l \to \infty} \overline{\mathcal{W}}(x_l,\Delta)}(z)~\mu_Z(dz) \\
&\leq \int_{\mathrm{supp} \; \mu_Z} \mathrm{1}_{\overline{\mathcal{W}}(x_0,\Delta)}(z)~\mu_Z(dz) \\
&= \mu_{Z}[\overline{\mathcal{W}}(x_0,\Delta)] = M_{\Delta}(x_0),
\end{align*}
where the second inequality holds due the outer semicontinuity of $\overline{W}$ and the monotonicity of the indicator function. Consequently, $M_{\Delta}$ is upper semicontinuous at $x_0$ for any $\Delta \in D$.

\medskip

By $U \subseteq \mathrm{int} \; F_Z$ and the arguments used in the proof of Lemma \eqref{LemmafDiff},
\begin{equation}
\label{LSCSum}
\mathrm{supp} \; \mu_Z \subseteq \bigcup_{\Delta \in D} \mathcal{W}(x,\Delta)
\end{equation}
holds for any $x \in U$. By $\mathcal{W}(x,\Delta_1) \cap \mathcal{W}(x,\Delta_2) = \emptyset$ for any $\Delta_1, \Delta_2 \in D$ satisfying $\Delta_1 \neq \Delta_2$, \eqref{LSCSum} implies
$$
\sum_{\Delta \in D} M_{\Delta}(x) = \sum_{\Delta \in D} \mu_Z[\mathcal{W}(x,\Delta)] = 1
$$
for any $x \in U$. Consequently, as $M_{\Delta}$ is upper semicontinuous at $x_0$ for any $\Delta \in D$, we obtain that
$$
M_{\Delta}(x) = 1 - \sum_{\Delta' \in D \setminus \lbrace \Delta \rbrace} M_{\Delta'}(x)
$$
is representable as a sum functions that are lower semicontinuous at $x_0$. Thus, $M_{\Delta}$ is continuous at $x_0$ for any $\Delta \in D$, which implies the continuity of
$$
Q_\mathbb{E}'(x) = c^\top + \sum_{\Delta \in D} M_{\Delta}(x) \Delta
$$
at $x_0$.
\end{proof}

\medskip

When working with the expected excess, the inner maximum may cause additional points of nondifferentiability.

\medskip

\begin{theorem} \label{PropEEDiff}
Assume $\mathrm{dom} \; f \neq \emptyset$, $\mu_Z \in \mathcal{M}^1_s$, and let $x_0 \in \mathrm{int} \; F_Z$ and $\eta \in \mathbb{R}$ be such that $\mu_Z[\mathcal{N}_{x_0} \cup \mathcal{L}(x_0,\eta)] = 0$, where
$$
\mathcal{L}(x_0, \eta) := \bigcup_{\hat{A}_{B} \in \mathcal{A}^\ast: \; \hat{q}_{B}^\top \hat{A}_B^{-1} \neq 0} \left\{ z \in \mathbb{R}^s \; | \; c^\top x_0 +  \hat{q}_{B}^\top \hat{A}_B^{-1}(Tx_0 + z) = \eta \right\}.
$$
Then $Q_{\mathrm{EE}_\eta}$ is continuously differentiable at $x_0$.
\end{theorem}

\medskip

\begin{proof}
Consider the mapping $g_\eta: \mathbb{R}^n \times \mathbb{R}^s \to \overline{\mathbb{R}}$ given by
$$
g_\eta(x,z) := \max \lbrace f(x,z) - \eta, 0 \rbrace,
$$
which is finite and Lipschitz continuous on $F$ by Lemma \ref{LemmaF}. Consider any fixed $z_0 \in \mathrm{supp} \; \mu_Z \setminus \big( \mathcal{N}_{x_0} \cup \mathcal{L}(x_0,\eta) \big)$. If $f(x_0,z_0) \neq \eta$, there is a neighborhood $U$ of $x_0$ such that either $g_{\eta}(x,z_0) = f(x,z_0) - \eta$ for all $x \in U$ or $g_{\eta}(x,z_0) = 0$ for all $x \in U$. In both cases $g_{\eta}(\cdot,z_0)$ is continuously differentiable at $x_0$ by Theorem \ref{PropExpDiff}.
\\
\\
Now consider the case where $f(x_0,z_0) = \eta$. The proof of Lemma \ref{LemmafDiff} shows that there is some $\hat{A}_B \in \mathcal{A}^\ast$ such that $(x_0,z_0) \in \mathrm{int} \; \mathcal{R}(\hat{A}_B)$. In particular, we have $c^\top x_0 + \hat{q}_{B}^\top \hat{A}_B^{-1}(Tx_0 + z_0) = \eta$ and $z_0 \notin \mathcal{L}(x_0,\eta)$ implies $\hat{q}_{B}^\top \hat{A}_B^{-1} = 0$. Thus, $\eta = 0$ and there is a neighborhood $V$ of $x_0$ such that $g_\eta(x,z_0) = \max \lbrace \hat{q}_{B}^\top \hat{A}_B^{-1}(Tx + z_0), 0\rbrace = 0$ for all $x \in V$. Hence, $g_\eta(\cdot,z_0)$ is continuously differentiable at $x_0$.
\\
\\
Invoking Lemma \ref{LemmaDiffInt} and the above considerations, the differentiability of $Q_{\mathrm{EE}_\eta}$ and the continuity of
$$
Q_{\mathrm{EE}_\eta}'(x) = \sum_{\Delta \in D} \mu_Z \big[ \overline{\mathcal{W}}(x,\Delta) \cap \lbrace z \in \mathrm{supp} \; \mu_Z \; | \; f(x,z) \geq \eta \rbrace \big] \Delta
$$
at $x_0$ can be shown by a straightforward extension of the arguments used in the proof of Theorem \ref{PropExpDiff}.
\end{proof}

\medskip

\begin{theorem} \label{PropSDDiff}
Assume $\mathrm{dom} \; f \neq \emptyset$, $\mu_Z \in \mathcal{M}^1_s$, and let $x_0 \in \mathrm{int} \; F_Z$ be such that $Q_\mathbb{E}(x_0) \neq 0$ and $\mu_Z[\mathcal{N}_{x_0} \cup \mathcal{L}(x_0,Q_\mathbb{E}(x_0))] = 0$. Then $Q_{\mathrm{SD}_\rho}$ is continuously differentiable at $x_0$ for any $\rho \in [0,1)$.
\end{theorem}

\medskip

\begin{proof}
Fix any $p \in [0,1)$. By Theorem \ref{PropExpDiff} and the definition of $Q_{\mathrm{SD}_\rho}$ it is sufficient to show differentiability of the mapping $x \mapsto Q_{\mathrm{EE}_{Q_\mathbb{E}(x)}}(x)$. Consider the function $g: \mathbb{R}^n \times \mathbb{R}^s \to \overline{\mathbb{R}}$ defined by
$$
g(x,z) := \max \lbrace f(x,z) - Q_\mathbb{E}(x), 0 \rbrace,
$$
which is finite and Lipschitz continuous on $F$ by Lemma \ref{LemmaF} and Theorem \ref{ThFiniteLipschitz}. Fix any $z_0 \in \mathrm{supp} \; \mu_Z \setminus \big( \mathcal{N}_{x_0} \cup \mathcal{L}(x_0,Q_\mathbb{E}(x_0)) \big)$ and suppose that $f(x_0,z_0) = Q_\mathbb{E}(x_0)$. By the proof of Lemma \ref{LemmafDiff} there is some $\hat{A}_B \in \mathcal{A}^\ast$ such that $(x_0,z_0) \in \mathrm{int} \; \mathcal{R}(\hat{A}_B)$. In particular, we have $\hat{q}_{B}^\top \hat{A}_B^{-1}(Tx_0 + z_0) = Q_\mathbb{E}(x_0)$ and $z_0 \notin \mathcal{L}(x_0,Q_\mathbb{E}(x_0))$ implies $\hat{q}_{B}^\top \hat{A}_B^{-1} = 0$. Hence, $Q_\mathbb{E}(x_0) = f(x_0,z_0) = 0$, which contradicts the assumptions.
\\
\\
Thus, $f(x_0,z_0) \neq Q_\mathbb{E}(x_0)$ and there is a neighborhood $U$ of $x_0$ such that either $g(x,z_0) = f(x,z_0) - Q_\mathbb{E}(x_0)$ for all $x \in U$ or $g(x,z_0) = 0$ for all $x \in U$. In both cases $g(\cdot,z_0)$ is continuously differentiable at $x_0$ by Theorem \ref{PropExpDiff}.
\\
\\
Consequently, the differentiability of $Q_{\mathrm{SD}_\rho}$ and the continuity of
$$
Q_{\mathrm{SD}_\rho}'(x) = Q'_\mathbb{E}(x) + \rho \sum_{\Delta \in D} \mu_Z \big[ \overline{\mathcal{W}}(x,\Delta) \cap \lbrace z \in \mathrm{supp} \; \mu_Z \; | \; f(x,z) \geq Q_\mathbb{E}(x) \rbrace \big] \Delta
$$
at $x_0$ can be shown by a straightforward extension of the arguments used in the proof of Theorem \ref{PropExpDiff}.
\end{proof}

\medskip

\begin{cor}
\label{CorDensity}
Assume $\mathrm{dom} \; f \neq \emptyset$ and that $\mu_Z \in \mathcal{M}^1_s$ is absolutely continuous with respect to the Lebesgue measure. Fix any $\eta \in \mathbb{R}$, then $Q_\mathbb{E}$ and $Q_{\mathrm{EE}_\eta}$ are continuously differentiable at any $x_0 \in \mathrm{int} \; F_Z$. Furthermore, for any $\rho \in [0,1)$, $Q_{\mathrm{SD}_\rho}$ is continuously differentiable at any $x_0 \in \mathrm{int} \; F_Z$ satisfying $Q_\mathbb{E}(x_0) \neq 0$.
\end{cor}

\medskip

\begin{proof}
$Q_\mathbb{E}$: Since $\mathcal{N}_{x_0}$ is a finite union of affine hyperplanes, i.e. a set with Lebesgue measure zero,  $\mu_Z[\mathcal{N}_{x_0}] = 0$ holds for all $x_0 \in \mathrm{int} \; F_Z$ and the statement is a direct consequence of Theorem \ref{PropExpDiff}.

\medskip

$Q_{\mathrm{EE}_\eta}$: By definition, $\mathcal{L}(x_0, \eta)$ is a finite union of affine hyperplanes, which implies $\mu_Z[\mathcal{N}_{x_0} \cup \mathcal{L}(x_0,\eta)] = 0$ for any $x_0 \in \mathrm{int} \; F_Z$ and Theorem \ref{PropEEDiff} is applicable.

\medskip

$Q_{\mathrm{SD}_\rho}$: For any fixed $x_0$, $\mathcal{L}(x_0, Q_\mathbb{E}(x_0))$ is a set of Lebesgue measure zero and the statement follows from Theorem \ref{PropSDDiff}.
\end{proof}

\medskip

The previous results give sufficient conditions for differentiability of the  objective function of problem \eqref{ELB}. In the presence of differentiability, necessary  optimality can be formulated in terms of directional derivatives.

\medskip

\begin{proposition}
\label{PropNecessaryCond}
Assume $\mathrm{dom} \; f \neq \emptyset$, $\mu_Z \in \mathcal{M}^p_s$, and $X \subseteq F_Z$. Furthermore, let $x_0 \in X$ be a local minimizer of problem \eqref{ELB} and assume that $Q_\mathcal{R}$ is differentiable at $x_0$. Then
\begin{equation}
\label{NecessaryCond}
Q'_\mathcal{R}(x_0)d \geq 0
\end{equation}
holds for any feasible direction
$$
d \in \mathcal{D}(x_0,X) := \lbrace d \in \mathbb{R}^n \; | \; \exists \varepsilon_0 > 0: \; x_0 + \varepsilon d \in X \; \forall \varepsilon \in [0,\varepsilon_0] \rbrace.
$$
\end{proposition}

\vspace{-15pt}

\begin{proof}
$\mathrm{dom} \; f \neq \emptyset$, $\mu_Z \in \mathcal{M}^p_s$ and $X \subseteq F_Z$ imply that $Q_\mathcal{R}$ is real-valued on $X$ by Corollary \ref{CorCont} below. For a proof of the necessity of \eqref{NecessaryCond} we refer to \cite[Proposition 2.1.2]{Bertsekas1999}.
\end{proof}

\medskip

\begin{cor}
Assume $\mathrm{dom} \; f \neq \emptyset$, $X \subseteq F_Z$, and let $\mu_Z \in \mathcal{M}^1_s$ be absolutely continuous with respect to the Lebesgue measure. Furthermore, assume
$$
-c^\top \notin \mathrm{conv} \; D = \mathrm{conv} \; \lbrace \hat{q}_B^\top \hat{A}_B^{-1} T \; | \; \hat{A}_B \in \mathcal{A}^\ast \rbrace,
$$
then any local minimizer of
\begin{equation}
\label{ExpModel}
\min_x \lbrace Q_\mathbb{E}(x) \; | \; x \in X \rbrace
\end{equation}
is an element of $X \setminus \mathrm{int} \; X$.
\end{cor}

\medskip

\begin{proof}
Suppose that $x_0 \in \mathrm{int} \; X$ is a local minimizer of \eqref{ELB}, then Corollary \ref{CorDensity} and Proposition \ref{PropNecessaryCond} yield $0 = Q_\mathbb{E}'(x_0)$ as $\mathcal{D}(x_0,X) = \mathbb{R}^n$. Invoking the proof of Theorem \ref{PropExpDiff} we have
$$
0 = Q_\mathbb{E}'(x_0) = c^\top + \sum_{\Delta \in D} \mu_Z[\mathcal{W}(x_0,\Delta)]\Delta
$$
and thus $-c^\top \in \mathrm{conv} \; D$, which contradicts the assumptions.
\end{proof}

\section{A stability result for bilevel stochastic linear problems} \label{SecStability}

The aim of this section is to establish a qualitative stability result for the bilevel stochastic linear problem \eqref{ELB} with respect to perturbations of the underlying probability measure. Taking into account that the support of the perturbed measure may differ from the original support, we shall assume that $X \times \mathbb{R}^s \subseteq F$ to ensure that the objective function of \eqref{ELB} remains well defined.

\medskip

Throughout this section, we shall consider the general case where $\mathcal{R}$ is law-invariant and there exists some $p \geq 1$ such that the restriction $\mathcal{R}|_{L^p(\Omega, \mathcal{F}, \mathbb{P})}$ is a real-valued convex risk measure. Furthermore, for the sake of notational simplicity, we assume that the probability space $(\Omega, \mathcal{F}, \mathbb{P})$ is atomless (cf. Remark \ref{RemAtomless} below). Then for any $x \in X$ and $\mu \in \mathcal{M}^p_s$, Lemma \ref{LemmaF} implies $(\delta_x \otimes \mu) \circ f^{-1} \in \mathcal{M}^p_1$ and the atomlessness ensures that there exists some $Y_{(x,\mu)} \in L^p(\Omega, \mathcal{F}, \mathbb{P})$ such that $\mathbb{P} \circ Y_{(x,\mu)}^{-1} = (\delta_x \otimes \mu) \circ f^{-1}$. Thus, we may consider the mapping $\mathcal{Q}_\mathcal{R}: X \times \mathcal{M}^p_s \to \mathbb{R}$, $$
\mathcal{Q}_\mathcal{R}(x,\mu) := \mathcal{R}[Y_{(x,\mu)}].
$$
Note that the specific choice of $Y_{(x,\mu)}$ does not matter due to the law-invariance of $\mathcal{R}$.

\medskip

\begin{rem} \label{RemAtomless}
The assumption that $(\Omega, \mathcal{F}, \mathbb{P})$ is atomless does not entail a loss of generality: We may just fix an arbitrary atomless probability space $(\overline{\Omega}, \overline{\mathcal{F}}, \overline{\mathcal{P}})$, consider a law-invariant convex risk-measure $\overline{\mathcal{R}}: L^p(\overline{\Omega}, \overline{\mathcal{F}}, \overline{\mathcal{P}}) \to \mathbb{R}$ and define an the restriction $\mathcal{R}|_{L^p(\Omega, \mathcal{F}, \mathbb{P})}$ via $\mathcal{R}[Y] = \overline{\mathcal{R}}[\overline{Y}]$, where $\overline{Y}$ is an arbitrary random variable in $L^p(\overline{\Omega}, \overline{\mathcal{F}}, \overline{\mathcal{P}})$ satisfying $\overline{\mathbb{P}} \circ \overline{Y}^{-1} = \mathbb{P} \circ Y^{-1}$.
\end{rem}

\medskip

Consider the parametric optimization problem
\leqnomode
\begin{equation}
\label{ParametricProblem}
\tag{$\mathrm{P}_\mu$}
\min_x \lbrace \mathcal{Q}(x,\mu) \; | \; x \in X \rbrace.
\end{equation}
As \eqref{ParametricProblem} may be nonconvex, we shall pay special attention to sets of local optimal solutions. For any open set $V \subseteq \mathbb{R}^n$ we introduce the optimal value function $\varphi_V: \mathcal{M}^p_s \to \overline{\mathbb{R}}$,
$$
\varphi_V(\mu) := \min_x \lbrace \mathcal{Q}(x,\mu) \; | \; x \in X \cap \mathrm{cl} \; V \rbrace,
$$
as well as the localized optimal solution set mapping $\phi_V: \mathcal{M}^p_s \rightrightarrows \mathbb{R}^n$,
$$
\phi_V(\mu) := \underset{x}{\mathrm{Argmin}} \lbrace \mathcal{Q}(x,\mu) \; | \; x \in X \cap \mathrm{cl} \; V \rbrace.
$$
It is well known that additional assumptions are needed when studying stability of local solutions.

\medskip

\begin{defn}
Given $\mu \in \mathcal{M}^p_s$ and an open set $V \subseteq \mathbb{R}^n$, $\phi_V(\mu)$ is called a \emph{complete local minimizing (CLM) set} of \eqref{ParametricProblem} w.r.t. $V$ if $\emptyset \neq \phi_V(\mu) \subseteq V$.
\end{defn}

\medskip

\begin{rem}
The set of global optimal solutions $\phi_{\mathbb{R}^n}(\mu)$ and any set of isolated minimizers are CLM sets. However, in general, sets of strict local minimizers may fail to be CLM sets (cf. \cite{Robinson1987}).
\end{rem}

\medskip

In the following, we shall equip $\mathcal{P}(\mathbb{R}^s)$ with the topology of weak convergence, i.e. the topology where a sequence $\{\mu_l\}_{l \in \mathbb{N}} \subset \mathcal{P}(\mathbb{R}^s)$ converges weakly to $\mu \in \mathcal{P}(\mathbb{R}^s)$, written $\mu_l \stackrel{w}{\rightarrow} \mu$, iff
\begin{equation*}
	\lim_{l \to \infty} \int_{\mathbb{R}^s} h(t)~\mu_l(dt) = \int_{\mathbb{R}^s} h(t)~\mu(dt)
\end{equation*}
holds for any bounded continuous function $h:\mathbb{R}^s \to \mathbb{R}$ (cf. \cite{Billingsley1968}). The example below shows that even $\varphi_{\mathbb{R}^n}$ may fail to be weakly continuous on the entire space $\mathcal{P}(\mathbb{R}^s)$.

\medskip	
	
\begin{ex}
The problem
\begin{equation*}
\min_x \left\{x + \int_{\mathbb{R}^s} z~\mu(dz) \; | \; 0 \leq x \leq 1 \right\}
\end{equation*}
arises from a bilevel stochastic linear problem, where $\Psi(x,z) = \{z\}$ holds for any $(x,z)$. Assume that $\mu = \mathbb{P} \circ Z^{-1} = \delta_0$ is the Dirac measure at $0$. Then the above problem can be rewritten as $\min_x \{x \; | \; 0 \leq x \leq 1\}$ and its optimal value is $0$.

\medskip

However, while the sequence $\mu_l := (1 - \frac{1}{l}) \delta_{0} + \frac{1}{l} \delta_l$ converges weakly to $\delta_0$, replacing $\mu$ with $\mu_l$ yields the problem
\begin{equation*}
\min_x \left\{x + 1 \; | \; 0 \leq x \leq 1 \right\},
\end{equation*}
whose optimal value is equal to $1$ for any $l \in \mathbb{N}$.
\end{ex}

\medskip

In the present work, we shall follow the approach of \cite{ClausKraetschmerSchultz2017} and confine the stability analysis to locally uniformly $\|\cdot\|^p$-integrating sets.

\medskip

\begin{defn}
A set $\mathcal{M} \subseteq \mathcal{M}^p_s$ is said to be \emph{locally uniformly $\|\cdot\|^p$-integrating} iff for any $\epsilon > 0$ there exists some open neighborhood $\mathcal{N}$ of $\mu$ w.r.t. the topology of weak convergence such that
\begin{equation*}
	\lim_{a \to \infty} \sup_{\nu \in \mathcal{M} \cap \mathcal{N}} \int_{\mathbb{R}^s \setminus a \mathbb{B}} \|z\|^p~\nu(dz) \leq \epsilon.
\end{equation*}
\end{defn}

A detailed discussion of locally uniformly $\|\cdot\|^p$-integrating sets is provided in \cite{FoellmerSchied2011}, \cite{KraetschmerSchiedZaehle2014}, \cite{KraetschmerSchiedZaehle2017}, and \cite{KraetschmerSchiedZaehle2012}. The following examples demonstrate the relevance of the concept.

\medskip

\begin{ex}
\label{ExLocUnifIntSets}
(a) Fix $\kappa, \epsilon > 0$. Then by \cite[Corollary A.47, (c)]{FoellmerSchied2011}, the set
\begin{equation*}
	\mathcal{M}(\kappa, \epsilon) := \left\{ \mu \in \mathcal{P}(\mathbb{R}^s) \; | \; \int_{\mathbb{R}^s} \|z\|^{p + \epsilon} \leq \kappa \right\}
\end{equation*}
of Borel probability measures with uniformly bounded moments of order $p + \epsilon$ is locally uniformly $\|\cdot\|^p$-integrating.

\medskip

(b) Fix any compact set $\Xi \subset \mathbb{R}^s$. By \cite[Corollary A.47, (b)]{FoellmerSchied2011}, the set
\begin{equation*}
	\{\mu \in \mathcal{P}(\mathbb{R}^s) \; | \; \mu[\Xi] = 1\}
\end{equation*}
of Borel probability measures whose support is contained in $\Xi$ is locally uniformly $\|\cdot\|^p$-integrating.

\medskip

(c) Any singleton $\lbrace \mu \rbrace \subset \mathcal{M}^p_s$ is locally uniformly $\|\cdot\|^p$-integrating by \cite[Lemma 5.2]{{KraetschmerSchiedZaehle2017}}.

\end{ex}

\medskip

\begin{theorem} \label{ThStability} Assume $\mathrm{dom} \; f \neq \emptyset$ and $X \times \mathbb{R}^s \subseteq F$. Let $\mathcal{M} \subseteq \mathcal{M}^{p}_{s}$ be locally uniformly $\|\cdot\|^{p}$-integrating, then

\smallskip

\begin{enumerate}
	\item[(a)] $\mathcal{Q}_\mathcal{R}|_{X \times \mathcal{M}}$ is real-valued and weakly continuous.
	
	\smallskip	
	
	\item[(b)] $\varphi_{\mathbb{R}^n}|_\mathcal{M}$ is weakly upper semicontinuous.
\end{enumerate}

\smallskip

In addition, assume that $\mu_0 \in \mathcal{M}$ is such that $\phi_V(\mu_0)$ is a CLM set of $P_{\mu_0}$ w.r.t. some open bounded set $V \subset \mathbb{R}^n$. Then the following statements hold true:

\smallskip

\begin{enumerate}
	\item[(c)] $\varphi_V|_\mathcal{M}$ is weakly continuous at $\mu_0$.
	
	\smallskip
	
	\item[(d)] $\phi_V|_\mathcal{M}$ is weakly Berge upper semicontinuous at $\mu_0$, i.e. for any open set $\mathcal{O} \subseteq \mathbb{R}^{n}$ with $\phi|_V(\mu_0) \subseteq \mathcal{O}$ there exists a weakly open neighborhood $\mathcal{N}$ of $\mu_0$ such that $\phi_V(\mu) \subseteq \mathcal{O}$ for all $\mu \in \mathcal{N} \cap \mathcal{M}$.	

	\smallskip
	
	\item[(e)] There exists some weakly open neighborhood $\mathcal{U}$ of $\mu_0$ such that $\phi_V(\mu)$ is a CLM set for \eqref{ParametricProblem} w.r.t. $V$ for any $\mu \in \mathcal{U} \cap \mathcal{M}$.
\end{enumerate}
\end{theorem}

\medskip

\begin{proof}
Fix any $x_0 \in X$. By Lemma \ref{LemmaF}, $f$ is Lipschitz continuous on $X \times \mathbb{R}^s$. Thus, there exists a constant $L > 0$ such that
\begin{equation*}
	|f(x,z)| \leq L \|z\| + L\|x-x_0\| + |f(x_0, 0)|
\end{equation*}
and the result follows from \cite[Corollary 2.4.]{ClausKraetschmerSchultz2017}.
\end{proof}

\medskip

\begin{rem}
The assumption $X \times \mathbb{R}^s \subseteq F$ is equivalent to $F = \mathbb{R}^n \times \mathbb{R}^s$ and holds if and only if there is some $y \in \mathbb{R}^m$ such that $Ay < 0$. By Gordan's Theorem (\cite{Gordan1873}), the latter holds iff $u = 0$ is the only nonnegative solution to $A^\top u = 0$. Under this condition, the feasible set of the lower level is full dimensional for any leader's decision $x$ and any parameter $z$.
\end{rem}

\medskip

If the underlying distribution is fixed, the assumptions of Theorem \ref{ThStability} (a) can be weakened significantly.

\medskip

\begin{cor}
\label{CorCont}
Assume $\mathrm{dom} \; f \neq \emptyset$ and $\mu_Z \in \mathcal{M}^p_s$. Then $Q_\mathcal{R}$ is real-valued and continuous on $F_Z$. In addition, assume that $X \subseteq F_Z$ is nonempty and compact, then problem \eqref{ELB} is solvable.
\end{cor}

\medskip

\begin{proof}
The set $\lbrace \mu_Z \rbrace$ is locally uniformly $\|\cdot\|^p$-integrating by Example \ref{ExLocUnifIntSets}. Thus, continuity of $Q_\mathcal{R}(\cdot) = \mathcal{Q}_\mathcal{R}(\cdot, \mu_Z)$ can be established as in the proof of Theorem \ref{ThStability} (a) and the solvability of \eqref{ELB} is a direct consequence of the compactness of $X$.
\end{proof}

\medskip

\begin{ex}
(a) The assumptions of Section \ref{SecModel} are fulfilled for the expected excess of order $p \geq 1$ given by
$$
\mathrm{EE}_\eta^p[\cdot] =  \mathbb{E}[\max \lbrace \cdot- \eta, 0\}^p]^{\frac{1}{p}},
$$
where $\eta \in \mathbb{R}$ is a fixed target level (cf. \cite[Example 6.22]{DentchevaRuszczynskiShapiro2014}). Thus, the mapping $Q_{\mathrm{EE}_\eta^p}$ is continuous on $F_Z$ under the assumptions of Corollary \ref{CorCont}.

\medskip

(b) The mean upper semideviation of order $p \geq 1$ given by
$$
\mathrm{SD}_\rho^p[\cdot] = \mathbb{E}[\cdot] + \rho \mathrm{EE}_{E[\cdot]}^p[\cdot]
$$
is a law-invariant coherent risk measure for any $\rho \in [0,1)$ by \cite[Example 6.20]{DentchevaRuszczynskiShapiro2014}. Thus, Corollary \ref{CorCont} gives sufficient conditions for continuity of $Q_{\mathrm{SD}_\rho^p}$.
\end{ex}

\medskip

\begin{rem}
All results of Sections \ref{SecStructure} and \ref{SecStability} can be easily extended to the pessimistic approach to bilevel stochastic linear optimization, where $f$ takes the form
$$
f(x,z) = c^\top x - \min_y \{- q^\top y \; | \; y \in \Psi(x,z)\}.
$$
\end{rem}

As any Borel probability measure is the weak limit of a sequence of measures having finite support, Theorem \ref{ThStability} justifies an approach where the true underlying measure is approximated by a sequence of finite discrete ones.

\section{Finite discrete distributions} \label{SecFiniteDiscreteDistributions}

Throughout this section, we shall assume that the underlying random vector $Z$ is discrete with a finite number of realizations $Z_1, \ldots, Z_K \in \mathbb{R}^s$ and respective probabilities $\pi_1, \ldots, \pi_K \in (0,1]$. Let $I$ denote the index set $\lbrace 1, \ldots, K \rbrace$, then $F_Z$ takes the form
$$
F_Z = \lbrace x \in \mathbb{R}^n \; | \; \forall k \in I \; \exists y \in \mathbb{R}^m: \; Ay \leq Tx + Z_k \rbrace.
$$
Suppose that $x_0 \in X$ is such that $\lbrace y \in \mathbb{R}^m \; | \; Ay \leq Tx_0 + Z_k \rbrace = \emptyset$ holds for some $k \in I$. Then the probability of $f(x_0,Z(\omega)) = \infty$ is a least $\pi_k > 0$, i.e. $x_0$ should be considered as infeasible for problem \eqref{ELB}. Consequently, $X \subseteq F_Z$ can be understood as an induced constraint. Note that $X \cap F_Z$ is a polyhedron if $X$ is a polyhedron.

\medskip

We shall show that for models involving the expectation, the expected excess or the mean upper semideviation, problem \eqref{ELB} can be reduced to a standard bilevel linear program.

\medskip

\begin{theorem}[Expectation] \label{ThFinDiscExp}
Assume $\mathrm{dom} \; f \neq \emptyset$ and let $X \subseteq F_Z$ be a polyhedron, then the risk neutral bilevel stochastic linear problem
\begin{equation*}
\min_x \left\{Q_{\mathbb{E}}(x) \; | \; x \in X \right\}
\end{equation*}
is equivalent to the optimistic bilevel linear program
\begin{equation}
\label{FinDiscrExp}
\min_x \left\{c^\top x + \min_{y_1, \ldots, y_K} \left\{\sum_{k \in I} \pi_k q^\top y_k \; | \; (y_1, \ldots, y_K) \in \Psi_{\mathbb{E}}(x) \right\} \bigg | \; x \in X \right\},
\end{equation}
where $\Psi_{\mathbb{E}}: \mathbb{R}^n \rightrightarrows \mathbb{R}^{Km}$ is given by
\begin{equation*}
\Psi_{\mathbb{E}}(x) :=  \underset{y_1, \ldots, y_K}{\mathrm{Argmin}} \; \left\{\sum_{k \in I} d^\top y_k \; | \; Ay_k \leq Tx + Z_k \; \forall k \in I \right\}.
\end{equation*}
\end{theorem}

\begin{proof}
We have
\begin{align*}
Q_{\mathbb{E}}(x) &= c^\top x + \sum_{k \in I} \pi_k f(x,Z_k) \\
&= c^\top x + \sum_{k \in I} \pi_k \min_{y_k} \lbrace q^\top y_k \; | \; y_k \in \Psi(x,Z_k) \rbrace \\
&= c^\top x + \min_{y_1, \ldots, y_k} \left\{ \sum_{k \in I} \pi_k q^\top y_k \; | \; y_k \in \Psi(x, Z_k) \; \forall k \in I \right\}
\end{align*}
and the result follows from $\Psi_\mathbb{E}(x) = \Psi(x,Z_1) \times \ldots \times\Psi(x,Z_K)$.
\end{proof}

\medskip

\begin{rem}
The proof of Theorem \ref{ThFinDiscExp} shows that the inner minimization problem in \eqref{FinDiscrExp} can be decomposed into $K$ problems of similar structure.
\end{rem}

\medskip

\begin{theorem}[Expected excess] \label{ThFinDiscExpExcess}
Assume $\mathrm{dom} \; f \neq \emptyset$ and let $X \subseteq F_Z$ be a polyhedron, then for any $\eta \in \mathbb{R}$, the risk-averse bilevel stochastic linear problem
\begin{equation*}
\min_x \left\{Q_{\text{EE}_{\eta}}(x) \; | \; x \in X \right\}
\end{equation*}
is equivalent to the optimistic bilevel linear program
\begin{equation}
\label{FinDiscrExpExcess}
\min_x \left\{ \min_{\substack{y_1, \ldots, y_K, \\ v_1, \ldots, v_K}} \left\{ \sum_{k \in I} \pi_k v_k \,| \, (y_1, \ldots, y_K, v_1, \ldots, v_K) \in \Psi_{\mathrm{EE}_{\eta}}(x) \right\} \, \bigg | \, x \in X \right\},
\end{equation}
where $\Psi_{\mathrm{EE}_{\eta}}: \mathbb{R}^n \rightrightarrows \mathbb{R}^{Km + K}$ is given by
{\small
\begin{equation*}
\Psi_{\mathrm{EE}_{\eta}}(x) :=  \underset{\genfrac{}{}{0pt}{}{y_1, \ldots, y_K,}{v_1, \ldots, v_K}} {\mathrm{Argmin}} \left\{ \sum_{k \in I} d^\top y_k \, | \, Ay_k \leq Tx + Z_k, \, v_k \geq 0, \, v_k \geq c^\top x + q^\top y_k - \eta \, \forall k \in I \right\}.
\end{equation*}
}
\end{theorem}

\begin{proof}
We have $Q_{\text{EE}_{\eta}}(x) = \sum_{k \in I} \pi_k g_k(x)$, where
\begin{align*}
g_k(x) &= \max\{0, f(x, Z_k) - \eta\} \\
&= \min_{v_k} \left\{ v_k \; | \; v_k \geq 0, \; v_k \geq f(x, Z_k) - \eta \right\} \\
&= \min_{v_k} \Big\{ v_k \; \Big| \; v_k \geq 0, \; v_k \geq c^\top x + \min_{y_k} \{ q^\top y_k \; | \; y_k \in \Psi(x, Z_k)\} - \eta \Big\} \\
&= \min_{y_k, v_k} \left\{ v_k \; | \; v_k \geq 0, \; v_k \geq c^\top x + q^\top y_k - \eta,\; y_k \in \Psi(x, Z_k) \right\}
\end{align*}
holds for any $(x,k) \in X \times I$. Thus,
\begin{align*}
&Q_{\text{EE}_{\eta}}(x) = \sum_{k=1}^K \pi_k \inf_{y_k, v_k} \left\{ v_k \; | \; v_k \geq 0, \; v_k \geq c^\top x + q^\top y_k - \eta, \; y_k \in \Psi(x,Z_k) \right\} \\
&= \inf_{\genfrac{}{}{0pt}{}{y_1, \ldots, y_K,}{v_1, \ldots, v_K}} \bigg\{ \sum_{k=1}^K \pi_k v_k \; | \; v_k \geq 0, \; v_k \geq c^\top x + q^\top y_k - \eta,\; y_k \in \Psi(x,Z_k) \; \forall k \in I \bigg\} \\
&= \inf_{\genfrac{}{}{0pt}{}{y_1, \ldots, y_K,}{v_1, \ldots, v_K}} \bigg\{\sum_{k=1}^K \pi_k v_k \; | \; (y_1, \ldots, y_K, v_1, \ldots, v_K) \in \Psi_{\text{EE}_{\eta}}(x) \bigg\},
\end{align*}
which completes the proof.
\end{proof}

\medskip

\begin{rem}
Let $\Psi_{\mathrm{EE}_{\eta, k}}: X \rightrightarrows \mathbb{R}^{m +1}$ be given by
$$
\Psi_{\mathrm{EE}_{\eta, k}}(x) := \underset{y_k, v_k}{\mathrm{Argmin}} \left\{ d^\top y_k \, | \, Ay_k \leq Tx + Z_k, \, v_k \geq 0, \, v_k \geq c^\top x + q^\top y_k - \eta \right\},
$$
then $\Psi_{\mathrm{EE}_{\eta}}(x)$ admits the representation
$$
\Psi_{\mathrm{EE}_{\eta}}(x) = \lbrace (y_1, \ldots, y_K, v_1, \ldots, v_K) \; | \; (y_k,v_k) \in \Psi_{\mathrm{EE}_{\eta, k}}(x) \; \forall k \in I \rbrace.
$$
Thus, the inner minimization problem in \eqref{FinDiscrExpExcess} decomposes into $K$ problems of similar structure.
\end{rem}

\medskip

\begin{theorem}[Mean upper semideviation] \label{ThFinDiscSD}
Assume $\mathrm{dom} \; f \neq \emptyset$ and let $X \subseteq F_Z$ be a polyhedron, then for any $\rho \in [0,1)$, the risk-averse bilevel stochastic linear problem
\begin{equation*}
\min_x \left\{Q_{\mathrm{SD}_{\rho}}(x) \; | \; x \in X \right\}
\end{equation*}
is equivalent to the optimistic bilevel linear program
\begin{equation}
\label{FinDiscrSD}
\min_x \left\{ c^\top x + \min_{\genfrac{}{}{0pt}{}{y_1, \ldots, y_K,}{v_1, \ldots, v_K}} \left\{ \begin{aligned} &(1 - \rho) \sum_{k \in I} \pi_k q^\top y_k + \rho \sum_{k \in I} \pi_k v_k \\ s.t. \; &(y_1, \ldots, y_K, v_1, \ldots, v_K) \in \Psi_{\mathrm{SD}_\rho}(x) \; \end{aligned} \right\} \; | \; x \in X \right\},
\end{equation}
where $\Psi_{\mathrm{SD}_\rho}: \mathbb{R}^n \rightrightarrows \mathbb{R}^{Km + K}$ is given by
{\small
$$
\Psi_{\mathrm{SD}_\rho}(x) := \underset{\genfrac{}{}{0pt}{}{y_1, \ldots, y_K,}{v_1, \ldots, v_K}}{\mathrm{Argmin}} \left\{ \sum_{k \in I} d^\top y_k \; | \; Ay_k \leq Tx + Z_k, \; v_k \geq q^\top y_k, \; v_k \geq \sum_{j \in I} \pi_j q^\top y_j \; \forall k \in I \right\}.
$$
}
\end{theorem}

\begin{proof}
By
\begin{align*}
Q_{\mathrm{SD}_\rho}(x) &= Q_\mathbb{E}(x) + \rho Q_{\mathrm{EE}_{Q_\mathbb{E}(x)}}(x) \\
&= Q_\mathbb{E}(x) + \rho c^\top x + \rho \sum_{k \in I} \pi_k \max \left\{0, f(x,Z_k) - Q_\mathbb{E}(x)\right\} \\
&= (1 - \rho)Q_\mathbb{E}(x) + \rho c^\top x + \rho \sum_{k \in I} \pi_k \max \left\{f(x,Z_k), Q_\mathbb{E}(x)\right\} \\
&= (1 - \rho)Q_\mathbb{E}(x) + \rho c^\top x + \rho \sum_{k \in I} \pi_k \min_{v_k} \lbrace v_k \; | \; v_k \geq f(x,Z_k), \; v_k \geq Q_\mathbb{E}(x) \rbrace
\end{align*}
and the representation of $Q_\mathbb{E}$ that was established in the proof of Theorem \ref{ThFinDiscExp}, we have
{\small
\begin{align*}
&Q_{\mathrm{SD}_\rho}(x) = c^\top x + (1 - p) \sum_{k \in I}\pi_k \min_{y_k} \left\{q^\top y_k \; | \; y_k \in \Psi(x, Z_k) \; \forall k \in I \right\} \\
&\textcolor{white}{Q_{\mathrm{SD}_\rho}(x) =} + \min_{v_1, \ldots, v_K} \left\{ \rho \sum_{k \in I} \pi_k v_k \; \Bigg{|} \; \begin{aligned} &v_k \geq \min_{y_k} \lbrace q^\top y_k \; | \; y_k \in \Psi(x,Z_k) \rbrace, \\ &v_k \geq \sum_{k \in I}\pi_k \min_{y_k} \left\{q^\top y_k \; | \; y_k \in \Psi(x, Z_k) \right\} \end{aligned} \right\} \\
&= c^\top x + \min_{\substack{y_1, \ldots, y_K, \\ v_1, \ldots, v_K}} \left\{ (1 - \rho) \sum_{k \in I} \pi_k q^\top y_k + \rho \sum_{k \in I} \pi_k v_k \, \Bigg{|} \, \begin{aligned} &y_k \in \Psi(x,Z_k), \, v_k \geq q^\top y_k &\forall k \in I \\ &v_k \geq \sum_{j \in I} \pi_j q^\top y_j &\forall k \in I \end{aligned} \right\},
\end{align*}
}
which completes the proof.
\end{proof}

\medskip

\begin{rem}
The inner minimization problem in \eqref{FinDiscrSD} does not decompose scenariowise due to the $K$ coupling constraints $v_k \geq \sum_{j \in I} \pi_j q^\top y_j$ for $k \in I$ in the description of $Q_{\mathrm{SD}_\rho}(x)$.
\end{rem}

\medskip

Finally, we shall consider models involving the Conditional Value at Risk.

\medskip

\begin{theorem}[Conditional Value at Risk]
Assume $\mathrm{dom} \; f \neq \emptyset$ and let $X \subseteq F_Z$ be a polyhedron, then for any $\alpha \in (0,1)$, the risk-averse bilevel stochastic linear problem
\begin{equation*}
\min_x \left\{Q_{\mathrm{CVaR}_\alpha}(x) \; | \; x \in X \right\}
\end{equation*}
is equivalent to
{\small
\begin{equation}
\label{FinDiscrCVaR}
\min_x \left\{ \min_{\eta \in \mathbb{R}} \left\{ \eta + \min_{\substack{y_1, \ldots, y_K, \\ v_1, \ldots, v_K}} \left\{ \begin{aligned} &\frac{1}{1- \alpha} \sum_{k \in I} \pi_k v_k \\ s.t. \; &(y_1, \ldots, y_K, v_1, \ldots, v_K) \in \Psi_{\mathrm{EE}_{\eta}}(x) \end{aligned} \right\} \right\} \, \Bigg{|} \, x \in X \right\}.
\end{equation}
}
\end{theorem}

\medskip

\begin{proof}
As
$$
Q_{\mathrm{CVaR}_{\alpha}}(x) = \min_{\eta \in \mathbb{R}} \left\{ \eta + \frac{1}{1 - \alpha} Q_{\mathrm{EE}_\eta}(x) \right\},
$$
the result follows directly from the representation of $Q_{\mathrm{EE}_\eta}(x)$ that was established in the proof Theorem \ref{ThFinDiscExpExcess}.
\end{proof}

\medskip

\begin{rem}
Every evaluation the objective function in \eqref{FinDiscrCVaR} corresponds to solving a bilevel linear problem with scalar upper level variable $\eta$.
\end{rem}

\section{A regularization scheme for bilevel linear problems}

In the setting of Theorems \ref{ThFinDiscExp}, \ref{ThFinDiscExpExcess} and \ref{ThFinDiscSD}, the risk-averse bilevel stochastic linear problem may be reformulated as a standard optimistic bilevel linear problem of the form
\begin{equation}
\label{GenOptBilevel}
\min_u \lbrace g^\top u + \min_w \lbrace h^\top w \; | \;  w \in \Psi(u) \rbrace \; | \; u \in U \rbrace,
\end{equation}
where $\Psi: \mathbb{R}^k \rightrightarrows \mathbb{R}^l$ is given by
$$
\Psi(u) = \underset{w}{\mathrm{Argmin}} \lbrace t^\top w \; | \; Ww \leq Bu + b \rbrace
$$
for vectors $g \in \mathbb{R}^k$, $h, t \in \mathbb{R}^l$ and $b \in \mathbb{R}^r$, matrices $W \in \mathbb{R}^{r \times l}$ and $B \in \mathbb{R}^{r \times k}$, and a nonempty polyhedron $U \subseteq \mathbb{R}^k$.

\medskip

We shall discuss a solution approach for \eqref{GenOptBilevel} that relies on replacing it with a regularized single level problem involving the Karush-Kuhn-Tucker (KKT) conditions of the lower level problem.

\medskip

\begin{theorem}[{cf. \cite[Theorem 3.7]{Henkel2014}}, \cite{Luderer1983}]
Assume that $\mathrm{Argmin}_w \lbrace t^\top w \; | \; w \in \Psi(u) \rbrace$ is nonempty for any $u \in U$. Then the following statements hold true:

\smallskip

\begin{enumerate}
\item[(a)] The optimal values of \eqref{GenOptBilevel} and
\begin{equation}
\label{OneLevelRef}
\min_{u, w} \lbrace g^\top u + h^\top w \; | \; u \in U, \; w \in \Psi(u) \rbrace
\end{equation}
coincide.

\smallskip
	
\item[(b)] $\overline{u}$ is a global minimizer of \eqref{GenOptBilevel} if and only if there exists some $\overline{w}$ such that $(\overline{u},\overline{w})$ is a global minimizer of \eqref{OneLevelRef}.	

\smallskip

\item[(c)] $\overline{u}$ is a local minimizer of \eqref{GenOptBilevel} if and only if there exists some $\overline{w}$ such that $(\overline{u},\overline{w})$ is a local minimizer of \eqref{OneLevelRef}.
\end{enumerate}
\end{theorem}

\medskip

\begin{proof}
By assumption, the mapping $\varphi_o: U \to \mathbb{R}$
$$
\varphi_o(u) := \min_w \lbrace h^\top w \; | \; w \in \Psi(u) \rbrace
$$
is well-defined and for any $\tilde{u} \in U$ there exists some $\tilde{w} \in \Psi(\tilde{u})$ such that $h^\top \tilde{w} = \varphi_o(\tilde{u})$. Furthermore, $\varphi_o(\tilde{u}) \leq h^\top w$ holds for any $w \in \Psi(\tilde{u})$, which implies (a), (b) and the "if" part of (c).

\medskip

To show the "only if" part of (c), suppose that $(\overline{u},\overline{w})$ is a local minimizer of \eqref{OneLevelRef}. Then there exist some $\epsilon > 0$ such that
\begin{equation}
\label{ProofOptEquiv1}
g^\top u + h^\top w \geq g^\top \overline{u} + h^\top \overline{w}
\end{equation}
holds for any $(u,w) \in B_\epsilon(\overline{u},\overline{w})$ satisfying $u \in U$ and $w \in \Psi(u)$. In particular, we have $g^\top \overline{u} + h^\top w \geq g^\top \overline{u} + h^\top \overline{w}$ for any $w \in \Psi(\overline{u}) \cap B_\epsilon(\overline{w})$, which implies that $\overline{w}$ is a local and thus global minimizer of the linear program
$$
\min_w \lbrace h^\top w \; | \; w \in \Psi(\overline{u}) \rbrace.
$$
Consider the mapping $M: U \rightrightarrows \mathbb{R}^l$ defined by
$$
M(u) := \underset{w}{\mathrm{Argmin}} \lbrace h^\top w \; | \; w \in \Psi(u) \rbrace = \lbrace w \; | \; Ww \leq Bu + b, \; h^\top w \leq \varphi_o(u) \rbrace.
$$
As $\varphi_o$ is Lipschitz continuous by Theorem \ref{TheoremKlatteRightHandSide} in the Appendix, Lipschitz continuity of $M$ follows from the same result. Suppose that $\overline{u}$ is not a local minimizer of \eqref{GenOptBilevel}, then there exist a sequence $\lbrace u_n \rbrace_{n \in \mathbb{N}}$ such that  and $u_n \in U$ and
\begin{equation}
\label{ProofOptEquiv2}
g^\top u_n + \varphi_o(u_n) < g^\top \overline{u} + \varphi_o(\overline{u})
\end{equation}
hold for any $n \in \mathbb{N}$ and we have $\lim_{n \to \infty} u_n = u$. The Lipschitz continuity of $M$ and $\overline{w} \in M(\overline{u})$ imply
$$
\lim_{n \to \infty} \; \inf_{w \in M(u_n)} \|w - \overline{w} \| = 0.
$$
Thus, there exists a sequence $\lbrace w_n \rbrace_{n \in \mathbb{N}}$ satisfying $\lim_{n \to \infty} w_n = \overline{w}$ and $w_n \in M(u_n)$ for all $n \in \mathbb{N}$. Consequently, by \eqref{ProofOptEquiv2}, there is some $N \in \mathbb{N}$ such that for any $n \geq N$, we have $(u_n,w_n) \in B_\epsilon(\overline{u},\overline{w})$ and
$$
g^\top u_n + h^\top w_n = g^\top u_n + \varphi_o(u_n) < g^\top \overline{u} + \varphi_o(\overline{u}) = g^\top \overline{u} + h^\top \overline{w},
$$
which contradicts \eqref{ProofOptEquiv1}. Thus, $\overline{u}$ is a local minimizer of \eqref{GenOptBilevel}.
\end{proof}

\medskip

Next, we use the KKT conditions of the lower level problem to replace \eqref{OneLevelRef} with the single-level problem
\begin{equation}
\label{KKTReform}
\min_{u, w, v} \left\{ g^\top u + h^\top w \; \Bigg| \; \begin{aligned}  &Ww \leq Bu + b, \; W^\top v = t, \; v \leq 0, \\ &v^\top (Ww - Bu - b) = 0, \; u \in U \end{aligned} \right\}.
\end{equation}
The relationship between bilevel problems and mathematical programs with complementarity constraints arising from the lower level KKT system has been investigated in \cite{DempeDutta2012}. In the special case of bilevel linear problems, the following holds:

\medskip

\begin{theorem}[{cf. \cite[Theorem 3.2]{DempeDutta2012}}]

\smallskip

\begin{enumerate}
\item[(a)] The optimal values of \eqref{OneLevelRef} and \eqref{KKTReform} coincide

\smallskip
	
\item[(b)] $(\overline{u}, \overline{w})$ is a global minimizer of \eqref{OneLevelRef} if and only if there exists some $\overline{v}$ such that $(\overline{u},\overline{w}, \overline{v})$ is a global minimizer of \eqref{KKTReform}.	

\smallskip

\item[(c)] $(\overline{u}, \overline{w})$ is a local minimizer of \eqref{OneLevelRef} if and only if $(\overline{u},\overline{w}, \overline{v})$ is a local minimizer of \eqref{KKTReform} for any $\overline{v} \leq 0$ satisfying $W^\top \overline{v} = t$ and $\overline{v}^\top (W\overline{w} - B\overline{u} - b) = 0$.
\end{enumerate}
\end{theorem}

\medskip

\begin{proof}
As the lower level problem is linear, its KKT conditions are necessary and sufficient for optimality. Thus, we have $w \in \Psi(u)$ if and only if there exists some $v \leq 0$ such that $W^\top v = t$ and $v^\top (Ww - Bu - b) = 0$, which implies (a), (b) and the "if" part of (c).

\medskip

To show the "only if" part of (c), let $(\overline{u},\overline{w}, \overline{v})$ be a local minimizer of \eqref{KKTReform} for any $\overline{v} \leq 0$ satisfying $W^\top \overline{v} = t$ and $\overline{v}^\top (W\overline{w} - B\overline{u} - b) = 0$ and suppose that $(\overline{u},\overline{w})$ is not a local minimizer of \eqref{OneLevelRef}. Then there exist sequences $\lbrace u_n \rbrace_{n \in \mathbb{N}} \subseteq U$ and $\lbrace w_n \rbrace_{n \in \mathbb{N}}$ such that $\lim_{n \to \infty} u_n = \overline{u}$, $\lim_{n \to \infty} w_n = \overline{w}$ and for any $n \in \mathbb{N}$ we have $w_n \in \Psi(u_n)$ and
\begin{equation}
\label{ProofKKTEquiv1}
g^\top u_n + h^\top w_n < g^\top \overline{u} + h^\top \overline{w}.
\end{equation}
As the mapping $\Lambda: \mathrm{gph} \; \Psi \rightrightarrows \mathbb{R}^r$ given by
\begin{equation}
\label{DefLambda}
\Lambda(u,w) := \lbrace v \in \mathbb{R}^r \; | \; W^\top v = t, \; v \leq 0, \; v^\top (Ww - Bu -b) = 0 \rbrace
\end{equation}
is outer semicontinuous by Lemma \ref{LemmaOuterSemicont} in the Appendix, there exists some $N \in \mathbb{N}$ such that
\begin{equation}
\label{ProofKKTEquiv2}
\Lambda(u_n,w_n) \subseteq \Lambda(\overline{u},\overline{w})
\end{equation}
holds for all $n \geq N$. Fix any converging sequence $\lbrace v_n \rbrace_{n \in \mathbb{N}}$ such that $v_n \in \Lambda(u_n, w_n)$ holds for any $n \in \mathbb{N}$. By \eqref{ProofKKTEquiv2} we have $\overline{v} = \lim_{n \to \infty} v_n \in \Lambda(\overline{u}, \overline{w})$. Thus, $(\overline{u},\overline{w},\overline{v})$ is a local minimizer of \eqref{KKTReform}. In particular, there exists some $\overline{N} \in \mathbb{N}$ such that
$g^\top u_n + h^\top w_n \geq g^\top \overline{u} + h^\top \overline{w}$ for all $n \geq \overline{N}$, which contradicts \eqref{ProofKKTEquiv1}.
\end{proof}

\medskip

It is known that often used regularity conditions as Mangasarian-Fromovitz constraint qualification or Slater's constraint qualification are violated at every feasible point of \eqref{KKTReform} (cf. \cite{ScheelScholtes2000}). To overcome the difficulties related with this property, we propose to replace \eqref{KKTReform} by
\leqnomode
\begin{equation}
\label{KKTRelax}
\tag*{$\mathrm{P(\varepsilon)}$}
\min_{u, w, v} \left\{ g^\top u + h^\top w \; \Bigg| \; \begin{aligned}  &Ww \leq Bu + b, \; W^\top v = t, \; v \leq 0, \\ &v^\top (Ww - Bu - b) \leq \varepsilon, \; u \in U \end{aligned} \right\}
\end{equation}
and solve this problem for $\varepsilon \downarrow 0$. This approach and its use to solve general mathematical programs with equilibrium constraints has been investigated in \cite{Scholtes2001}. For the special case of the bilevel linear optimization problem \eqref{OneLevelRef} we can prove the following result:

\medskip

\begin{theorem} \label{TheoremLinearScholtes}
Let $(\overline{u}, \overline{w}, \overline{v})$ be an accumulation point of a sequence $\lbrace (u_n, w_n, v_n) \rbrace_{n \in \mathbb{N}}$ of local minimizers of problem $\mathrm{P(}\varepsilon_n\mathrm{)}$ for $\varepsilon_n\downarrow 0$. Then $(\overline{u}, \overline{w})$ is a local minimizer of \eqref{OneLevelRef}.
\end{theorem}

\medskip

\begin{proof}
Without loss of generality, we may assume that $\lbrace (u_n, w_n, v_n) \rbrace_{n \in \mathbb{N}}$ converges. Suppose that $(\overline{u}, \overline{w})$ is not a local minimizer of \eqref{OneLevelRef}. Then, since $U$ is a polyhedron and $\mathrm{gph} \; \Psi$ is polyhedral (cf. \cite[Theorem 3.1]{Dempe2002}), i.e. equal to the union of a finite number of polyhedra, there exist a direction $(d_u, d_w) \in \mathbb{R}^k \times \mathbb{R}^l$ and a sequence $\alpha_m \downarrow 0$ such that $\overline{u} + \alpha_m d_u \in U$, $\overline{w} + \alpha_m d_w \in \Psi(\overline{u} + \alpha_m d_u)$ and
\begin{equation}
\label{ProofScholtes1}
g^\top (\overline{u} + \alpha_m d_u) + h^\top (\overline{w} + \alpha_m d_w) < g^\top \overline{u} + h^\top \overline{w}
\end{equation}
hold for any $m \in \mathbb{N}$. As the mapping $\Lambda$ defined by \eqref{DefLambda} is outer semicontinuous, there exists a constant $N \in \mathbb{N}$ such that $\Lambda(\overline{u} + \alpha_m d_u, \overline{w} + \alpha_m d_w) \subseteq \Lambda(\overline{u},\overline{w})$ for any $m \geq N$. In particular, there exists some vertex $\tilde{v}$ of $\Lambda(\overline{u},\overline{w})$ such that $\tilde{v}$ is a vertex of $\Lambda(\overline{u} + \alpha_m d_u, \overline{w} + \alpha_m d_w)$ for any $m \geq N$. We shall prove that there exists some $\overline{N} \in \mathbb{N}$ such that
\begin{equation}
\label{ProofScholtes2}
W(w_n + \alpha_m d_w) - B(u_n + \alpha_m d_u) - b \leq 0
\end{equation}
holds for any $m, n\geq \overline{N}$.

\medskip

For any $i \in \lbrace 1, \ldots, r \rbrace$ with $e_i^\top \tilde{v} < 0$ and any $m \geq N$, we have
$$
e_i^\top \big(W \overline{w} - B \overline{u} - b + \alpha_m (W d_w - Bd_u)\big) =  e_i^\top \big(W(\overline{w} + \alpha_m d_w) - B(\overline{u} + \alpha_m d_u) - b \big)= 0.
$$
As $e_i^\top (W \overline{w} - B \overline{u} - b) = 0$ and $\alpha_m > 0$, this implies $e_i^\top (W d_w - B d_u) = 0$. Furthermore, since $(u_n,w_n)$ is feasible for $\mathrm{P(\varepsilon_n)}$, we conclude that
$$
e_i^\top \big(W(w_n + \alpha_m d_w) - B(u_n + \alpha_m d_u) - b \big) \leq 0
$$
for any $m, n \in \mathbb{N}$.

\medskip

Similarly, for any $i \in \lbrace 1, \ldots, r \rbrace$ such that $e_i^\top \tilde{v} = 0$ and $e_i^\top (W \overline{w} - B \overline{u} - b) = 0$, we obtain $e_i^\top (W d_w - B d_u) \leq 0$ and thus
$$
e_i^\top \big(W(w_n + \alpha_m d_w) - B(u_n + \alpha_m d_u) - b \big) \leq 0
$$
for any $m,n \in \mathbb{N}$.

\medskip

Finally, for any $i \in \lbrace 1, \ldots, r \rbrace$ such that $e_i^\top \tilde{v} = 0$ and $e_i^\top (W \overline{w} - B \overline{u} - b) < 0$, the existence of some $\overline{N} \in \mathbb{N}$ such that
$$
e_i^\top \big(W(w_n + \alpha_m d_w) - B(u_n + \alpha_m d_u) - b \big) \leq 0
$$
for any $m,n \geq \overline{N}$ follows from the continuity of the mapping
$$
(u, w, \alpha) \mapsto W(w + \alpha d_w) - B(u + \alpha d_u) - b.
$$
By the above considerations, we have
\begin{equation}
\label{ProofScholtes3}
\tilde{v}^\top \big(W(w_n + \alpha_m d_w) - B(u_n + \alpha_m d_u) - b \big) = \tilde{v}^\top \big(W w_n - Bu_n - b \big)
\end{equation}
and $\lim_{n \to \infty} \tilde{v}^\top (W w_n - Bu_n - b) = 0$. Furthermore, as $\varepsilon_n \downarrow 0$, $(u_n,w_n,v_n)$ is feasible for $\mathrm{P(}\varepsilon_{n'}\mathrm{)}$ for any $n' \geq n$. Thus, we may assume that
\begin{equation}
\label{ProofScholtes4}
\tilde{v}^\top \big(W w_n - Bu_n - b \big) \leq \frac{\varepsilon_n}{2}
\end{equation}
holds for any $n \in \mathbb{N}$ without loss of generality. \eqref{ProofScholtes2}, \eqref{ProofScholtes3} and \eqref{ProofScholtes4} imply that $(u_n + \alpha_m d_u, w_n + \alpha_m d_w, \tilde{v})$ is feasible for $\mathrm{P(}\varepsilon_n\mathrm{)}$ for any $m,n \geq \overline{N}$.

\medskip

Fix $n \geq \overline{N}$. We shall prove that for any $\lambda \in (0,1]$, there is some $M_\lambda \geq \overline{N}$ such that
\begin{align*}
&\lambda (u_n + \alpha_m d_u, w_n + \alpha_m d_w, \tilde{v}) + (1 - \lambda) (u_n,w_n,v_n) \\
= \; &(u_n + \lambda \alpha_m d_u, w_n + \lambda \alpha_m d_w, \lambda \tilde{v} + (1-\lambda) v_n)
\end{align*}
is feasible for $\mathrm{P(}\varepsilon_n\mathrm{)}$ whenever $m \geq M_\lambda$. As $\lim_{m \to \infty} (1-\lambda) \lambda \alpha_m v_n^\top \big( Wd_w -Bd_u \big) = 0$, there exists some $M_\lambda \geq \overline{N}$ such that
$$
(1-\lambda) \lambda \alpha_m v_n^\top \big( Wd_w -Bd_u \big) \leq \lambda \frac{\varepsilon_n}{2}
$$
for all $m \geq M_\lambda$. By \eqref{ProofScholtes3}, \eqref{ProofScholtes4}, and the feasibility of $(u_n,w_n,v_n)$ for $\mathrm{P(}\varepsilon_{n}\mathrm{)}$, we have
\begin{align*}
&\big(\lambda \tilde{v} + (1-\lambda) v_n\big)^\top  \big(W(w_n + \lambda \alpha_m d_w) - B(u_n + \lambda \alpha_m d_u) - b \big) \\
= \; &\lambda \tilde{v}^\top \big( W(w_n + \lambda \alpha_m d_w) - B(u_n + \lambda \alpha_m d_u) - b \big) \\
+ \; &(1-\lambda) v_n^\top \big( W w_n - B u_n - b \big) + (1-\lambda) \lambda \alpha_m v_n^\top \big( Wd_w -Bd_u \big) \\
\leq \; &\lambda \frac{\varepsilon_n}{2} + (1- \lambda)\varepsilon_n + \lambda \frac{\varepsilon_n}{2} = \varepsilon_n
\end{align*}
for any $m \geq M_\lambda$ and feasibility follows from the linearity of the remaining restrictions.

\medskip

As \eqref{ProofScholtes2} implies $g^\top d_u + h^\top d_w < 0$,
$$
g^\top (u_n + \lambda \alpha_m d_u) + h^\top (w_n + \lambda \alpha_m d_w) < g^\top u_n + h^\top w_n
$$
holds for any $\lambda \in (0,1]$ and $m \geq M_\lambda$, which, by
$$
\lim_{m \to \infty} \lim_{\lambda \to 0} \lambda (u_n + \lambda \alpha_m d_u, w_n + \lambda \alpha_m d_w, \lambda \tilde{v} + (1-\lambda) v_n) = (u_n, w_n, v_n),
$$
yields a contradiction to the local optimality of $(u_n,w_n,v_n)$ for $\mathrm{P(}\varepsilon_{n}\mathrm{)}$.
\end{proof}

\medskip

\begin{rem} Let $(\overline{u}, \overline{w}, \overline{v})$ be an accumulation point of a sequence $\lbrace (u_n, w_n, v_n) \rbrace_{n \in \mathbb{N}}$ of global minimizers of problem $\mathrm{P(}\varepsilon_n\mathrm{)}$ for $\varepsilon_n \downarrow 0$. Then $(\overline{u}, \overline{w})$ is a global minimizer of \eqref{OneLevelRef} (see the ideas in the proof of Theorem 2.1 in \cite{DempeFranke2016} in combination with \cite{DempeDutta2012}).
\end{rem}

\section{Appendix} \label{SecAppendix}

We shall recall some technical results used throughout the paper.

\medskip

\begin{theorem}[{\cite[Theorem 4.2]{KlatteThiere1995}}]\label{TheoremKlatteRightHandSide}
If $D$ positive semidefinite, the set-valued mapping $C: \mathbb{R}^{k} \rightrightarrows \mathbb{R}^{m}$ given by
\begin{equation*}
	C(t) := \underset{y}{\mathrm{Argmin}} \{y^{\top}Dy + d_0^\top y \; | \; Ay \leq t\}
\end{equation*}
is Lipschitz continuous on $\mathrm{dom} \; C := \{t \in \mathbb{R}^k \; | \; C(t) \neq \emptyset\}$, i.e. there exists a constant $\Lambda > 0$ such that $d_{\infty}(C(t),C(t')) \leq \Lambda\|t-t'\|$ holds for any $t, t' \in \mathrm{dom} \; C$.
\end{theorem}

\medskip

The following result is a well-known direct consequence of Lebesgue's Dominated Convergence Theorem:

\medskip

\begin{lem} \label{LemmaDiffInt}
Let $\mu$ be a Borel-probability measure on $\mathbb{R}^s$, $V \subseteq \mathbb{R}^n \times \mathbb{R}^s$ open, and $g: V \to \mathbb{R}$ such that the following conditions are satisfied:

\smallskip

\begin{enumerate}
\item[(a)] $g(\cdot,z)$ is differentiable at $x_0 \in V_\mu := \lbrace x \; | \; (x,z) \in V \; \forall z \in \mathrm{supp} \; \mu \rbrace$ for $\mu$-almost all $z \in \mathbb{R}^s$ and the derivative $g'(x_0,z)$ is measurable with respect to $z$.

\smallskip

\item[(b)] There exists a neighborhood $U \subseteq V_\mu$ of $x_0$ such that
\begin{enumerate}
\item[(i)] the integral $\int_{\mathbb{R}^s} g(x,z)~\mu(dz)$ is well-defined and finite for all $x \in U$ and
\item[(ii)] there is an integrable function $m: U \to \mathbb{R}$ such that $|e(x,z)| \leq m(z)$ holds for all $x \in U \setminus \lbrace x_0 \rbrace$ and $\mu$-almost all $z \in \mathbb{R}^s$, where
$$
e(x,z) = \frac{1}{\|x-x_0\|} \Big( g(x,z) - g(x_0,z) - \nabla_x g(x_0,z)(x-x_0) \Big).
$$
\end{enumerate}
\end{enumerate}
Then $h: V_\mu \to \overline{\mathbb{R}}$,
$
h(x) = \int_{\mathbb{R}^s} g(x,z)~\mu(dz)
$
is differentiable at $x_0$ and
$$
h'(x_0) = \int_{\mathbb{R}^s} \nabla_x g(x_0,z)~\mu(dz).
$$
\end{lem}

\vspace{-15pt}

\begin{proof}
Set
$$
\varepsilon(x) := \frac{1}{\|x-x_0\|} \Big( h(x) - h(x_0) - \int_{\mathbb{R}^s} \nabla_x g(x_0,z)(x-x_0)~\mu(dz) \Big).
$$
By assumption, we have $\lim_{x \to x_0} |e(x,z)| = 0$ for $\mu$-almost all $z \in \mathbb{R}^s$ and Lebesgue's Dominated Convergence Theorem implies
$$
\lim_{x \to x_0} |\varepsilon(x)| \leq \lim_{x \to x_0} \int_{\mathbb{R}^s} |e(x,z)|~\mu(dz) = \int_{\mathbb{R}^s} \lim_{x \to x_0} |e(x,z)|~\mu(dz) = 0,
$$
which completes the proof.
\end{proof}

\medskip

\begin{lem} \label{LemmaOuterSemicont}
Let $\mathcal{C} \subseteq \mathbb{R}^{k \times s}$ be closed, then the set-valued mapping $\mathcal{T}: \mathbb{R}^k \rightrightarrows \mathbb{R}^l$,
$$
\mathcal{T}(t) = \lbrace z \in \mathbb{R}^s \; | \; (t,z) \in \mathcal{C} \rbrace
$$
is outer semicontinuous (cf. \cite{RockafellarWets2009}), i.e. $\limsup_{t \to t_0} \; \mathcal{T}(t) \subseteq \mathcal{T}(t_0) \; \forall t_0 \in \mathbb{R}^k$.
\end{lem}

\medskip

\begin{proof}
By definition of the outer limit, $z \in \limsup_{t \to t_0} \; \mathcal{T}(t)$ holds if and only if there are sequences $\lbrace t_n \rbrace_{n \in \mathcal{N}} \subset \mathbb{R}^k$ and $\lbrace z_n \rbrace_{n \in \mathcal{N}} \subset \mathbb{R}^s$ such that
$$
\lim_{n \to \infty} t_n = t_0, \; \lim_{n \to \infty} z_n = z \; \; \text{and} \; \; z_n \in \mathcal{T}(t_n) \; \forall n \in \mathbb{N}.
$$
For any such sequences we have $(t_n,z_n) \in \mathcal{C}$ for all $n \in \mathcal{N}$ and thus $(t_0,z) \in \mathcal{C}$. Consequently, $z \in \mathcal{T}(t_0)$.
\end{proof}

\section*{Acknowledgement}

The second author thanks the Deutsche Forschungsgemeinschaft for its support via the Collaborative Research Center TRR 154.

\end{document}